\documentclass{article}

\usepackage{amsmath}
\usepackage{amssymb}
\usepackage{amsthm}
\usepackage{graphicx}
\usepackage[latin1]{inputenc}
\usepackage[english]{babel}
\usepackage[T1]{fontenc}
\usepackage{float}
\usepackage{placeins}

\usepackage{xcolor}
\usepackage{comment}

\usepackage{latexsym,color,graphicx,enumerate}
\usepackage{hyperref}

\newtheorem{thm}{Theorem}[section]
\newtheorem{prop}[thm]{Proposition}
\newtheorem{lemma}[thm]{Lemma}
\newtheorem{rem}[thm]{Remark}
\newtheorem{defi}[thm]{Definition}

\newtheorem{nota}{Notation}[section]

\makeatletter

\newcommand*{\plim}[1][]{%
	\if\relax\detokenize{#1}\relax
	\def\next{\qopname\relax m{lim}}%
	\else
	\def\next{\qopname\newmcodes@ m{#1-lim}}%
	\fi
	\next
}

\newcommand*{\psum}[1][]{%
	\DOTSB
	\if\relax\detokenize{#1}\relax\else
	\operatorname{#1-}\mkern-\thinmuskip
	\fi
	\sum@\slimits@
}

\makeatother

\newcommand{\R}{\mathbb{R}}             
\newcommand{\N}{\mathbb{N}}             
\newcommand{\C}{\mathbb{C}}             

\newcommand {\clr} {\color{red}}

\numberwithin{equation}{section}

\title{On uniqueness of radial potentials for given Dirichlet spectra with distinct angular momenta}

\author{Damien Gobin, Beno\^it Gr\' ebert, Bernard Helffer  and Fran{\c{c}}ois Nicoleau}

\begin{document}

\maketitle


\begin{abstract}
We consider an inverse spectral problem for radial Schr\"odinger operators with singular potentials. 
First, we show that the knowledge of the Dirichlet spectra for infinitely many angular momenta~$\ell$ 
satisfying a M\"untz-type condition uniquely determines the potential. 
Next, in a neighborhood of the zero potential, we prove local uniqueness from two Dirichlet spectra 
associated with distinct angular momenta in the cases 
\((\ell_1,\ell_2) = (0,1)\,, \ (1,2)\) and  \((0,3)\)\,.
Our approach relies on an explicit analysis of the associated singular differential equation, 
combined with the classical Kneser--Sommerfeld formula. 
These results sharpen a theorem of Carlson-Shubin~(1994) and confirm, in the linearized setting and for these configurations, a conjecture originally formulated by Rundell and Sacks~(2001).
\end{abstract}

\section{Introduction}

Inverse spectral problems for Schr\" odinger operators play a central role in quantum mechanics and mathematical physics, particularly in determining unknown potentials from spectral measurements. A classical example arises when the Schr\" odinger equation is considered in three spatial dimensions with a spherically symmetric potential in the unit ball of $\R^3$. In such cases, a standard separation of variables in spherical coordinates reduces the original partial differential equation to a family of one-dimensional singular Sturm-Liouville problems of the form
\begin{equation} \label{eq:radial-Schro}
	- \frac{d^2u}{dr^2} + \left( \frac{\ell(\ell+1)}{r^2} + q(r) \right) u = \lambda u\,, \qquad r \in (0,1)\,,
\end{equation}
where $\ell \in \mathbb{N}$ is the angular momentum quantum number, and the radial potential $q$ is assumed to be real-valued and square-integrable on $(0,1)$. This equation is supplemented with a Dirichlet boundary condition at $r = 1$ and a regularity condition at the origin, typically of the form $u(r) = O(r)$ as $r \to 0\,$, which ensures square-integrability of the solution and reflects the physical finiteness of the wavefunction near the origin.

\vspace{0.1cm}\noindent
Such reduced radial problems are not only mathematically rich but also physically relevant. They appear in a variety of contexts including quantum scattering in bounded domains, the study of acoustic modes in stellar interiors, and in zonal decompositions of Laplace-type operators on spheres. See, for instance,~\cite{Carlson97,RuSa01}.

\vspace{0.1cm}\noindent
A foundational result in regular inverse spectral theory is the Borg--Levinson theorem, which asserts that a potential can be uniquely recovered from the knowledge of the Dirichlet spectrum together with the associated norming constants, (that is, the Neumann data of the associated eigenfunctions; see below for further details).
In the setting of singular Sturm--Liouville equations such as~\eqref{eq:radial-Schro}, Carlson~\cite{Carlson97} extended this result to include centrifugal singularities, establishing uniqueness under suitable conditions on the spectrum and norming data.

\vspace{0.1cm}\noindent
However, these classical results rely crucially on the availability of norming constants. In contrast, the present work investigates a different question, physically more relevant: can one determine the potential $q$ uniquely from the knowledge of the Dirichlet spectra corresponding to distinct values of $\ell$, \emph{without} any norming constants?

\vspace{0.1cm}\noindent
Before addressing this question, let us recall some spectral properties of the underlying operator. 
For each fixed angular momentum $\ell \in \mathbb{N}$ and potential $q \in L^2(0,1)$, 
the Dirichlet spectrum is given by the zeros of an entire function (the so-called regular solution; see Appendix~\ref{app:A}),
and consists of a countable sequence of simple, real eigenvalues $\{\lambda_{\ell,n}(q)\}_{n\geq 1}$ tending to infinity.
These properties are well established in the literature~\cite{CarlShu94,Carlson97} and form the basis for both direct and inverse spectral analysis in the radial setting.

\vspace{0.1cm}\noindent
This observation naturally raises the question of whether combining spectral data from distinct angular momenta can lead to uniqueness. Indeed, the spectrum corresponding to a single value of $\ell$ does not determine the potential uniquely: the associated isospectral set is locally infinite-dimensional in $L^2(0,1)$~\cite{Carlson97,GR,PT,Ser07}. By contrast, Carlson and Shubin~\cite{CarlShu94} showed that combining spectral data from two distinct angular momenta $\ell_1 \neq \ell_2$ significantly reduces the ambiguity. In particular, they proved that the corresponding isospectral set has finite local dimension when $\ell_2 - \ell_1$ is odd.

\vspace{0.1cm}\noindent
Rundell and Sacks~\cite{RuSa01} went further and conjectured that the potential is uniquely determined by the Dirichlet spectra corresponding to any two distinct values of $\ell$, regardless of the parity of their difference. This reflects the fact that the angular dependence introduced via the centrifugal term $\ell(\ell+1)/r^2$ encodes genuinely distinct spectral information for each angular momentum channel, which can be leveraged to resolve the inverse problem even in the absence of norming constants.


\vspace{0.1cm}\noindent
Before turning to the more delicate case of two angular momenta, let us first consider the situation where Dirichlet spectra are known for infinitely many values of~$\ell\,$. Our first theorem is a global uniqueness result and concerns the case where one knows 
the Dirichlet spectra for infinitely many angular momenta~$\ell$ satisfying a M\"untz-type condition. In fact, this result is a direct consequence of Theorem~2.1 established in~\cite{Ra99} in the framework of scattering theory; see also~(\cite{DaNi16}, Theorem 1.4)  for local uniqueness results and~\cite{Go18} for an extension to magnetic fields. The proof of Theorem~\ref{thm:global} is given in Appendix~\ref{app:A} , where we establish a correspondence 
between Dirichlet spectral data and scattering phases.

\begin{thm}[Global uniqueness]
\label{thm:global}
Let $(\ell_k)_{k \geq 1}$ be a strictly increasing sequence of positive integers such that
\[
\sum_{k=1}^\infty \frac{1}{\ell_k} = +\infty\,.
\]
Then the potential \( q \in L^2(0,1) \) is uniquely determined by the Dirichlet spectra 
\(\{\lambda_{\ell_k,n}(q)\}_{k,n\geq 1}\)\,.
\end{thm}

\vspace{0.1cm}\noindent
Having established this global uniqueness result in the case of infinitely many spectra, 
we now turn to the more delicate situation involving only finitely many angular momenta. 
Because of its singular structure and the nonlinear dependence of the spectrum on the potential, the full inverse problem poses significant analytical challenges. A natural first step is to consider its linearization around a reference potential, typically $q \equiv 0$. Let $\psi_{\ell,n}(x)$ denote the $n$-th normalized eigenfunction associated with angular momentum $\ell$ and potential $q$. Then the Fr\'echet derivative of the corresponding eigenvalue with respect to $q$ is given by
\begin{equation}\label{differentielle}
	D_q \lambda_{\ell,n}(q)\, \zeta = \int_0^1 \zeta(x)\, \psi_{\ell,n}(x)^2\, dx\,,
\end{equation}
where $\zeta$ denotes a perturbation of the potential, (see \cite{RuSa01}). The norming constants are defined by $\gamma_n:= (\psi_{\ell,n}'(1))^2$.  In the linearized setting, the requirement that two potentials $q$ and $\check{q}$ yield the same eigenvalues for two angular momenta $\ell_1 \neq \ell_2$ implies that $\zeta = \check{q} - q$ must be orthogonal to the family of squared eigenfunctions corresponding to both $\ell_1$ and $\ell_2$, evaluated at $q=0$.

\vspace{0.1cm}\noindent
In the unperturbed case $q \equiv 0$, the eigenfunctions of~\eqref{eq:radial-Schro} can be expressed explicitly in terms of Bessel functions of the first kind. Specifically, for each $\ell \in \mathbb{N}$\,, the eigenfunctions take the form (up to normalization)
\[
\psi_{\ell,n}(r) := \sqrt{\frac{\pi r}{2}}\, J_\nu(j_{\nu,n} r)\,, \qquad \nu = \ell + \tfrac{1}{2}\,, \quad n \geq 1\,,
\]
where $J_\nu$ denotes the Bessel function of order $\nu$, and $j_{\nu,n}$ is its $n$-th positive zero. These functions satisfy Dirichlet boundary conditions at $r = 1$, and regularity at the origin is ensured by the behavior $J_\nu(r) \sim r^\nu$ as $r \to 0$\,.

\vspace{0.1cm}\noindent
This leads to the problem of determining whether the span of the functions
\[
\left\{ \psi_{\ell,n}^2(x) \,:\, n \geq 1,\ \ell \in \{\ell_1, \ell_2\} \right\}
\]
is dense in $L^2(0,1)$. Establishing this completeness property in the linearized setting is a first step in order to prove the local uniqueness of the potential near $q = 0\,$.

\vspace{0.1cm}\noindent
Even in the more favorable case where the angular momentum $\ell$ takes infinitely many values, 
it is not known whether the vector space spanned by the functions $\psi_{\ell,n}^2(x)$ 
is dense in \(L^2(0,1)\)\,, 
despite the fact that the problem is expected to be highly overdetermined. 
Although global uniqueness can now be derived by other methods, see Appendix~\ref{app:A}, 
the following result remains of independent interest.

\begin{thm}[Completeness]
\label{thm:completude}
Let $(\ell_k)_{k \geq 1}$ be a strictly increasing sequence of positive integers such that
\[
\sum_{k=1}^\infty \frac{1}{\ell_k} = +\infty\,.
\]
Then the family
\[
\bigl\{ \psi_{\ell_k,n}^2(x) : k \geq 1,\ n \geq 1 \bigr\}
\]
is complete in \(L^2(0,1)\)\,.
\end{thm}

\noindent
The proof of this theorem is given in Section 3.2.

\vspace{0.2cm}\noindent
Independently of the completeness result, the asymptotic behavior of the eigenvalues imposes further constraints on admissible perturbations. Specifically, as shown in \cite[Proposition~3.1]{Ser07} (see also \cite[Eq.~(2.10)]{RuSa01}), for each fixed angular momentum \( \ell \), the following expansion holds uniformly on bounded subsets of \( L^2(0,1) \)\,:
\begin{equation}\label{asymptoticsvps}
\lambda_{\ell,n}(q) = \left(n + \tfrac{\ell}{2}\right)^2 \pi^2 + \int_0^1 q(x)\,dx - \ell(\ell+1) + \tilde{\lambda}_{\ell,n}(q)\,,
\end{equation}
with \( (\tilde{\lambda}_{\ell,n}(q) \in \ell_{\R}^2(\mathbb{N}) \)\,. Thus, if two potentials $q$ and $\check{q}$ produce identical eigenvalues for large $n$, their difference $\zeta = \check{q} - q$ must satisfy
\[
\int_0^1 \zeta(x)\,dx = 0\,.
\]
This condition reflects the fact that the mean of the potential appears as the subleading term in the asymptotic expansion and must therefore be preserved under any isospectral deformation.

\vspace{0.2cm}\noindent
We conclude this section with the main result of the paper. It shows that, in three specific cases, the Dirichlet spectra corresponding to two distinct angular momenta uniquely determine the potential in a neighborhood of the zero potential  in the $L^2(0,1)$ topology.

\begin{thm}[Local uniqueness near the zero potential]
	\label{thm:main-uniqueness}
	For \( (\ell_1,\ell_2) \in \{(0,1), (1,2), (0,3)\} \)\,, the knowledge of the Dirichlet spectra corresponding to angular momenta \( \ell_1 \) and \( \ell_2 \) uniquely determines the potential \( q \in L^2(0,1) \) in a  $L^2$-neighborhood of the zero potential.
\end{thm}

\begin{rem}
The case \( (\ell_1,\ell_2) = (0,2) \) is also investigated in this paper. 
We show that the differential of the spectral map (see Section~2 for its definition) 
is injective at \( q = 0 \)\,. We conjecture that this differential is an isomorphism.
If this is the case, Theorem \ref{thm:main-uniqueness} would also hold for 
\( (\ell_1,\ell_2) = (0,2) \) by the local inversion theorem. 
We leave this question as an open problem.
\end{rem}

\vspace{0.1cm}\noindent
Our approach is based on a detailed analysis of the spectral problem associated with~\eqref{eq:radial-Schro} for varying values of $\ell$\,, in combination with the classical Kneser--Sommerfeld expansion. This representation plays a central role in establishing connections between the spectral data corresponding to different angular momenta, and in proving completeness results that support uniqueness in the linearized regime. We believe that this uniqueness result holds for any pair of distinct angular momenta \( (\ell_1, \ell_2) \)\,.

\section{Strategy of the Proof}

\subsection{The spectral map} 

We define a spectral map involving the renormalized eigenvalues \( \widetilde{\lambda}_{\ell,n}(q) \) defined by the asymptotic expansion~\eqref{asymptoticsvps} and associated with two distinct angular momenta. We then prove that this spectral map is an immersion at the zero potential and consequently locally injective at this point.

\vspace{0.1cm}\noindent
Let \( \ell_1 \neq \ell_2 \) be two fixed non-negative integers.  
Consider the spectral map
\[
\mathcal{S}_{\ell_1,\ell_2} \,: \, L^2(0,1) \longrightarrow \mathbb{R} \times \ell_{\R}^2(\mathbb{N}) \times \ell_{\R}^2(\mathbb{N})\,,
\]
defined by
\begin{equation}\label{eq:spectral-map-two-ell}
	\mathcal{S}_{\ell_1,\ell_2}(q) = \left( \int_0^1 q(x)\,dx\,,\; \big( \widetilde{\lambda}_{\ell_1,n}(q) \big)_{n \geq 1}\,,\; \big( \widetilde{\lambda}_{\ell_2,n}(q) \big)_{n \geq 1} \right)\,.
\end{equation}

\medskip

\vspace{0.1cm}\noindent
The following proposition, proved in \cite[Theorems~2.3 and~4.1]{Ser07} and based on
\eqref{differentielle}, states that the spectral map is real-analytic and
provides an explicit expression for its Fr\'echet differential at the zero
potential.

\vspace{0.2cm}\noindent
Before stating the proposition, we introduce the following notation.

\begin{nota}
	The spherical Bessel function $j_\ell$ is defined by
	\[
	j_\ell(z) = \sqrt{\frac{\pi z}{2}}\, J_{\nu}(z)\,,
	\qquad \nu=\ell+\tfrac12\,,
	\]
	where $J_\nu$ denotes the Bessel function of the first kind. For $n\in\mathbb{N}$\,,
	we set
	\[
	g_{\ell, n}(x)
	= \frac{j_\ell(j_{\nu,n}x)}{\|\, j_\ell(j_{\nu,n}\,\cdot)\,\|_{L^2(0,1)}}\,.
	\]
\end{nota}

\begin{prop}\label{thm:analyticity-spectral-map-zero}
	Let \( \ell_1 \neq \ell_2 \) be two fixed non-negative integers. 
	The spectral map
	\[
	\mathcal{S}_{\ell_1,\ell_2} : L^2(0,1) \to \mathbb{R} \times \ell_{\mathbb{R}}^2(\mathbb{N}) \times \ell_{\mathbb{R}}^2(\mathbb{N})
	\]
	is real-analytic. Moreover, for every \( \zeta \in L^2(0,1) \)\,, its Fr\'echet
	differential at the zero potential is given by
	\[
	d_0 \mathcal{S}_{\ell_1,\ell_2}(\zeta)
	= \left(
	\langle \zeta, 1\rangle\,, \quad
	\bigl( \langle \zeta, g_{\ell_1,n}^2 -1\rangle \bigr)_{n \geq 1}, \quad
	\bigl( \langle \zeta, g_{\ell_2,n}^2 -1 \rangle \bigr)_{n \geq 1}
	\right)\,,
	\]
	where \(\langle \cdot,\cdot\rangle\) denotes the scalar product in \(L^2(0,1)\)\,.
\end{prop}

\vspace{0.2cm}\noindent
We now state the linearized uniqueness result in terms of the spectral map and its differential.

\begin{thm}[Injectivity of the differential of the spectral map]
	\label{thm:differential-injective}
	Let \( \ell_1 \neq \ell_2 \) be two distinct non-negative integers, and consider the spectral map
	\[
	\mathcal{S}_{\ell_1,\ell_2} : L^2(0,1) \to \mathbb{R} \times \ell_{\mathbb{R}}^2(\mathbb{N}) \times \ell_{\mathbb{R}}^2(\mathbb{N})\,.
	\]
	Then the Fr\'echet differential at the zero potential $q=0$
	\[
	d_0 \mathcal{S}_{\ell_1,\ell_2} \,:\, L^2(0,1) \to \mathbb{R} \times \ell_{\mathbb{R}}^2(\mathbb{N}) \times \ell_{\mathbb{R}}^2(\mathbb{N})\,,
	\]
	is injective, provided that  \( (\ell_1,\ell_2) \in \{(0,1)\,,\, (0,2)\,,\, (1,2)\,,\, (0,3)\} \)\,.
\end{thm}

\vspace{0.2cm}
\noindent
The proof of Theorem~\ref{thm:differential-injective} is given in Sections~7-9.

\vspace{0.2cm}\noindent
As a byproduct of our analysis, we obtain the following theorem in the case $\ell_1 = 0$ or $1$ and arbitrary $\ell_2$\,.

\begin{thm}[Finite-dimensional kernel]
\label{thm:finite-kernel-ell1-0}
Let $\ell_1 =0$ or $1$\,, and $\ell_2 \geq 1$\,. Consider
\[
\mathcal{S}_{\ell_1,\ell_2} : L^2(0,1) \longrightarrow \mathbb{R} \times \ell^2_{\mathbb{R}}(\mathbb{N}) \times \ell^2_{\mathbb{R}}(\mathbb{N})\,.
\]
Then the Fr\'echet differential of $\mathcal{S}_{\ell_1,\ell_2}$ at the zero potential $q=0$ has a finite-dimensional kernel.
\end{thm}

\begin{proof}[Sketch of proof]
Our approach shows that any element of this kernel satisfies a homogeneous linear ODE on $(0,1)$, whose order and coefficients depend only on $\ell_2$\,. The cases $(\ell_1,\ell_2)=(0,1)$\,, $(0,2)$\,, $(1,2)$ and $(0,3)$ are treated in detail in this paper, and the extension to arbitrary $\ell_2$ follows without additional difficulty. Consequently, the kernel is spanned by finitely many fundamental solutions and is therefore finite-dimensional.
\end{proof}

\begin{rem}
This result extends Theorem~1 of Shubin Christ~\cite{Sh} in the case $(\ell_1,\ell_2)=(0,1)$\,,
as well as the result of Carlson--Shubin~\cite[Theorem~1.1]{CarlShu94} in the case where
$\ell_1-\ell_2$ is odd, at the level of the Fr\'echet differential at the zero potential $q=0$\,.
In these works, the authors proved that the Fr\'echet differential
$d_q\mathcal{S}_{\ell_1,\ell_2}$ has a finite-dimensional kernel for arbitrary potentials $q$\,.
We conjecture that the same conclusion holds for all pairs of distinct nonnegative integers
$\ell_1\neq \ell_2$\,.
\end{rem}

\vspace{0.2cm}\noindent
We conclude this section with the following result, whose proof is given in Appendix~\ref{app:B}.

\begin{thm}[Closed range]\label{Closed range}
Let \( (\ell_1,\ell_2) \in \{(0,1)\,,\, (1,2)\,,\, (0,3)\} \)\,.
Then the differential of the spectral map at the origin,
\( d_0 \mathcal{S}_{\ell_1,\ell_2} \)\,, has closed range. More precisely, in the case \( (\ell_1,\ell_2) = (0,1) \)\,,
it is surjective, and its range coincides with the entire space
\[
\mathbb{R}
\times \ell^2_{\mathbb{R}}(\mathbb{N})
\times \ell^2_{\mathbb{R}}(\mathbb{N})\,.
\]
\end{thm}

\vspace{0.1cm}\noindent
Consequently, in the case $(0,1)$\,, the standard local inverse function theorem
in Banach spaces, together with Theorem~\ref{thm:differential-injective} and
Theorem~\ref{Closed range}, implies Theorem~\ref{thm:main-uniqueness}.
More precisely, in this case the spectral map provides a local
real-analytic coordinate system.

\vspace{0.2cm}\noindent
The cases $(1,2)$ and $(0,3)$ are slightly different. For this reason, we recall below the following
local injectivity result, which is a direct consequence of the mean value theorem
and the open mapping theorem (see, for instance,~\cite{Abraham88}, Theorem 2.5.10).

\begin{prop}[Local injectivity]\label{prop:local-injectivity}
Let $X$ and $Y$ be Banach spaces, and let
\[
\mathcal{S} : U \subset X \longrightarrow Y
\]
be a $\mathcal C^1$ map defined on an open neighborhood $U$ of a point $x_0 \in X$\,.
Assume that the Fr\'echet differential $d_{x_0} \mathcal{S} \,:\,  X \to Y$
is injective and has closed range.
Then there exists a neighborhood $V \subset U$ of $x_0$ such that
$\mathcal{S}$ is injective on $V$.
\end{prop}

\vspace{0.2cm}\noindent
Consequently, Theorem~\ref{thm:main-uniqueness} follows from the previous
proposition together with Theorem~\ref{thm:differential-injective} and
Theorem~\ref{Closed range}.

\section{The Kneser--Sommerfeld formula}

The so-called Kneser--Sommerfeld expansion is a classical identity involving Bessel functions and their zeros. It appears in Watson's treatise~\cite{Wa44} (Section~15.42), where it is stated as a Fourier--Bessel series identity. However, as first noted by Buchholz in 1947~\cite{Buch47a}, and more recently clarified by Martin~\cite{Martin21}, the formula given in Watson is incorrect: it is missing an integral term on the right-hand side.

\vspace{0.1cm}\noindent
This correction was independently rediscovered in 1981 by Kobayashi~\cite{Ko81}, who analyzed the same identity and showed that Watson's version could not be valid in general. However, a proper historical credit for the correction belongs to Buchholz, whose 1947 paper already pointed out the failure and provided a corrected structure.

\vspace{0.1cm}\noindent
In his detailed analysis, Martin~\cite{Martin21} confirms that a different version  - found in earlier works by Kneser~\cite{Kneser07} for $\nu=0$, Sommerfeld~\cite{Sommerfeld12} for integer values of $\nu\,$, and Carslaw~\cite{Carslaw14} - is in fact the correct identity. Surprisingly, Watson does not include this valid form, despite its mathematical consistency and symmetry.

\vspace{0.1cm}\noindent
The corrected version of the Kneser--Sommerfeld expansion reads:
\begin{equation}
	\sum_{n=1}^{\infty}
	\frac{J_\nu(x j_{\nu,n}) J_\nu(X j_{\nu,n})}
	{(z^2 - j_{\nu,n}^2) \left[ J'_\nu(j_{\nu,n}) \right]^2}
	= \frac{\pi}{4 J_\nu(z)} J_\nu(xz) \left[ J_\nu(z) Y_\nu(Xz) - Y_\nu(z) J_\nu(Xz) \right]\,,
	\label{eq:KS-correct}
\end{equation}
valid for \( 0 \leq x \leq X \leq 1 \)\,, $ \nu \in \R$\,,  $ z \not= j_{\nu,n}$\,, where \( J_\nu \) and \( Y_\nu \) are the Bessel functions of the first and second kinds, respectively, and \( j_{\nu,n} \) denotes the \( n \)-th positive zero of \( J_\nu \)\,.

\vspace{0.1cm}\noindent
The identity (\ref{eq:KS-correct})  plays a fundamental role in the analysis of our inverse spectral problem.

\subsection{An integral Green identity}

In the next subsection, we will present a first completeness theorem for the inverse problem under consideration. The argument relies solely on the moment formula~\eqref{moment} together with the M\"untz--Sz\'asz theorem, and yields uniqueness under a rather restrictive vanishing condition on the moments.

\vspace{0.1cm}\noindent
So, we begin by considering a function $\zeta \in L^2(0,1)$ that is orthogonal to the family of squared eigenfunctions associated with a fixed angular momentum $\ell \in \mathbb{N}$ and vanishing potential. That is,
\begin{equation} \label{eq:orthogonality}
	\int_0^1 x\, \zeta(x) \left[ J_\nu(j_{\nu,n} x) \right]^2 dx = 0\,, \qquad \text{for all } n \geq 1\,,
\end{equation}
where $\nu = \ell + \tfrac{1}{2}$ and $j_{\nu,n}$ denotes the $n$-th positive zero of the Bessel function $J_\nu$\,.

\vspace{0.1cm}\noindent
Our goal is to exploit the structure of the family $\left\{ \left[J_\nu(j_{\nu,n} x)\right]^2 \right\}_{n \geq 1}$ using the Kneser--Sommerfeld expansion~\eqref{eq:KS-correct}, and to derive an additional integral identity satisfied by $\zeta\,$.

\vspace{0.2cm}\noindent
To that end, we consider a weighted sum of the orthogonality relations~\eqref{eq:orthogonality} over all $n \geq 1$\,. By exchanging the order of summation and integration (justified, for instance, by Fubini's theorem), we obtain for all $z$ not equal to any zero $j_{\nu,n}$:
\begin{equation} \label{eq:zeta-sum}
	\int_0^1 x\, \zeta(x)\, J_\nu(xz) \left[ J_\nu(z)\, Y_\nu(xz) - Y_\nu(z)\, J_\nu(xz) \right] dx = 0\,.
\end{equation}
The expression inside the integral corresponds, up to a multiplicative constant, to the kernel appearing in the corrected Kneser--Sommerfeld expansion on the diagonal $x = X$ multiplied by  $J_\nu(z)$. By a standard continuity argument, this identity extends to all complex values of \( z \)\,.

\begin{rem}
The integral kernel appearing in~\eqref{eq:zeta-sum} coincides with the gradient with respect to $q$  of the so-called regular solution 
evaluated at $q=0$ and $x=1$ (see~\cite{Ser07}, Proposition 2.2). 
This observation is consistent with Lemma~A.3 of the present paper.
\end{rem}

\subsection{A first completeness result}

\vspace{0.2cm}\noindent
In this subsection we prove Theorem~\ref{thm:completude}, the completeness result stated in the introduction. Our strategy is to establish first an integral identity, which will then allow us to deduce a moment condition.  
Combining this with the classical M\"untz--Sz\'asz theorem ultimately leads to the desired completeness result, which appears as a direct consequence of the final theorem in this section.

\vspace{0.2cm}\noindent
Recall that, as \( x \to 0 \)\,, the Bessel functions admit the following asymptotic behavior (see \cite{Lebedev72}):
\[
J_\nu(x) \sim \frac{1}{\Gamma(\nu + 1)} \left( \frac{x}{2} \right)^\nu\,, \quad
Y_\nu(x) \sim -\frac{\Gamma(\nu)}{\pi} \left( \frac{2}{x} \right)^\nu \quad \text{for } \nu > 0\,.
\]
Thus, taking the limit \( z \to 0 \) in (\ref{eq:zeta-sum}), we obtain the following closed-form expression:
\begin{equation}\label{moment}
	\int_0^1 x \ \zeta(x)\,  \left(1 - x^{2\nu} \right) dx = 0\,.
\end{equation}

\vspace{0.2cm}\noindent
We can now state a first completeness result based on the classical M\"untz--Sz\'asz theorem (\cite{Muntz12, Muntz14, Szasz16}).  
It is worth emphasizing that, in the following theorem, no zero-mean condition is required for the function $\zeta$.

\begin{thm}[A completeness result]
	\label{thm:muntz}
	Let $\zeta \in L^2(0,1)$ be a real-valued function satisfying (\ref{eq:orthogonality}) for an infinite increasing sequence $\nu_k= \{\ell_k + \tfrac12\}$ of positive half-integers such that
	\[
	\sum_{k=1}^\infty \frac{1}{\ell_k} = +\infty\,.
	\]
	Then $\zeta = 0$ almost everywhere in $(0,1)$\,.
\end{thm}

\begin{proof}
	Let \(\{ \ell_k \} \subset \mathbb{N}^*\) be the sequence of positive integers appearing above, with \(\sum_{k=1}^\infty \frac{1}{\ell_k} =+ \infty\)\,, and apply identity~\eqref{moment} with  \(\nu_k = \ell_k + \tfrac12\). This yields, as $k \to + \infty$\,,
	\[
	\int_0^1 x\, \zeta(x)\, dx = 0\,.
	\]
	Plugging this back into~\eqref{moment}, we obtain the moment identities
	\[
	\int_0^1 \zeta(x)\, x^{2\ell_k + 2}\, dx = 0 \quad \text{for all } k \in \mathbb{N}\,.
	\]
	By the classical M\"untz--Sz\'asz theorem, this implies that \(\zeta = 0\) almost everywhere in \((0,1)\)\,.
\end{proof}

\vspace{0.2cm}\noindent
As explained at the beginning of this section, Theorem~\ref{thm:muntz} immediately yields Theorem~\ref{thm:completude}, thereby completing the proof of the first completeness result.

\section{Transformation Operators}

Transformation operators play a central role in our inverse spectral analysis.  
For the case \(\ell = 1\), such operators were first introduced by Guillot and Ralston~\cite{GR} in their study of radial Schr\"odinger operators. These operators were also used by C.~Shubin Christ in \cite{Sh}. Carlson and Shubin~\cite{CarlShu94} extended this construction to more general settings, and Rundell and Sacks~\cite{RuSa01} developed a systematic framework valid for all integer orders \(\ell \ge 1\).  
Their method rewrites inner products involving Bessel kernels in terms of trigonometric ones by means of suitable \emph{index-reduction} operators.

\vspace{0.2cm}\noindent
We recall below two key lemmas due to Rundell and Sacks \cite{RuSa01} (also used by Serier \cite{Ser07}). These lemmas provide the basic tools for reducing the Bessel case to the trigonometric one and will be crucial for the fine analysis of isospectral sets developed later.

\begin{nota}
	Let \(\nu = \ell +\tfrac12\) for \(\ell \ge 0\) and \(x\in[0,1]\)\,. Define
	\[
	\Phi_\ell (x) \;=\; \frac{\pi x}{2}\,J_{\nu}(x)^2\,,
	\qquad
	\Psi_\ell (x) \;=\; -\,\frac{\pi x}{2}\,J_{\nu}(x)\,Y_{\nu}(x)\,,
	\]
	where \(J_{\nu}\)\,, \(Y_{\nu}\) are the ordinary Bessel functions.
\end{nota}

\begin{lemma}\label{lemmeRS0}[Rundell--Sacks \cite{RuSa01}]
	\label{rundel}
	For each integer \(\ell \ge 1\) and \(x\in(0,1)\), define the \emph{index-reduction operator}
	\begin{equation}\label{defSl}
		S_\ell [f](x)
		\;=\;
		f(x)
		-4\ell \,x^{2\ell -1}
		\int_x^1 t^{-2\ell}\,f(t)\,dt\,.
	\end{equation}
	Then:
	\begin{enumerate}[(i)]
		\item  \(S_\ell \) is bounded on \(L^2(0,1)\).
		\item Its Hilbert adjoint is
		\begin{equation}
			S_\ell ^*[g](x)
			\;=\;
			g(x)
			-\frac{4\ell}{x^{2\ell}}
			\int_0^x t^{2\ell -1}\,g(t)\,dt\,.
		\end{equation}

		\item It induces a Banach isomorphism
		\[
		L^2(0,1)\;\xrightarrow{\sim}\;
		\bigl\{x\mapsto x^{2\ell }\bigr\}^{\perp}\,,
		\]
		with inverse
		\begin{equation}
			A_\ell [g](x)
			\;=\;
			g(x)
			-\frac{4 \ell }{x^{2\ell +1}}
			\int_0^x t^{2\ell }\,g(t)\,dt\,.
		\end{equation}
		\item If \(g=S_\ell [f]\), then \(f\) and \(g\) satisfy
		\begin{equation}\label{ODEfonda}
			f^{(2\ell)}(x) + \frac{4\ell}{x}\, f^{(2\ell-1)}(x) \;=\; g^{(2\ell)}(x)\,.
		\end{equation}
		\item The operators commute pairwise:
		\[
		S_\ell \,S_m \;=\; S_m\,S_\ell \,,
		\quad  \forall\,\ell,m\in\N\,.
		\]
		\item The functions \(\Phi_\ell \) and \(\Psi_\ell \) satisfy
		\begin{align}
			\Phi_\ell  &= -\,S_\ell ^*\bigl[\Phi_{\ell -1}\bigr]\,,
			&
			\Psi_\ell  &= -\,S_\ell ^*\bigl[\Psi_{\ell -1}\bigr],
			\label{eq_reduc_indice}\\
			\Phi_\ell ' &= -\,A_\ell \bigl[\Phi_{\ell -1}'\bigr]\,,
			&
			\Psi_\ell ' &= -\,A_\ell \bigl[\Psi_{\ell -1}'\bigr]\,.
			\label{eq_reduc_indice_derive}
		\end{align}
		
	\end{enumerate}
\end{lemma}

\vspace{0.3cm}\noindent
We will also need the following complementary result (see Lemma 3.4 in \cite{RuSa01}).
\begin{lemma}\label{propadjoint}
If for $\ell \geq 1$\,,  $f$ and $g \in L^2(0,1)$\,, it holds 
$g = S_\ell^*[f]$ then, in $\mathcal D'(0,1)$\,, we have
\[ g^{(2\ell+1)}(x) + \frac{4\ell}{x} g^{(2\ell)} (x) = f^{(2\ell +1)}(x)\,.\]
\end{lemma}

\begin{proof}
We follow the strategy of \cite{RuSa01} for the proof of \eqref{ODEfonda}. 
We start from 
\begin{equation}\label{eqg}
 g(x) = f(x) - 4 \ell x^{-2\ell} \int_0^x s^{2\ell-1} f(s) \, ds\,,
 \end{equation}
and differentiate once to obtain
\begin{equation}\label{eqgprime}
g'(x) = f'(x) + 8\ell^2 x^{-2\ell -1} \int_0^x s^{2\ell-1} f(s) \, ds - \frac{4\ell}{x} f(x)\, .
 \end{equation}
We then eliminate the integral term by considering  $\frac{2\ell }{x} \times \eqref{eqg} + \eqref{eqgprime}$ :
\[ \frac{2\ell}{x} g(x) + g'(x) = \frac{2\ell }{x} f(x) + f'(x)  -\frac{4\ell}{x} f(x)\,.\]
This leads to 
\[ 2\ell g(x) + x g'(x) = - 2\ell f(x) + xf'(x)\, .\]
Differentiating again, we get
\[ (2\ell+1)g'(x) + xg''(x) = (-2\ell + 1) f'(x) + xf''(x)\,,\]
and iterating $k$ times we obtain
\[ (2\ell +k)g^{(k)}(x) +x g^{(k+1)}(x) = (-2\ell+k)f^{(k)}(x) + x f^{(k+1)}(x)\,.\]
Taking $k = 2\ell$ and dividing by $x$ achieves the proof of the lemma.
\end{proof}

\vspace{0.2cm}\noindent
We now consider the composite operator \(T_\ell\), obtained by composing the index-reduction operators \(S_1,\dots,S_\ell\), which carries Bessel kernels to trigonometric ones.

\begin{lemma}\label{lemmeRS}[Rundell--Sacks \cite{RuSa01}, Serier \cite{Ser07}]
	\label{lemme_bessel_sinus}
	Let us define $T_\ell$ by  $T_1=S_1$ and  for $\ell\geq 2$ by 
	\[
	T_\ell  \;=\; (-1)^{\ell +1}\,S_\ell \,S_{\ell -1}\cdots S_1\,.
	\]
	Then:
	\begin{enumerate}[(i)]
		\item \(T_\ell\) is a bounded and injective operator on \(L^2(0,1)\). For any \(\zeta\in L^2(0,1)\) and \(z\in\C\),
		\begin{align}
			\int_0^1\bigl[2\,\Phi_\ell (z t)-1\bigr]\,\zeta(t)\,dt
			&=\int_0^1\cos(2z t)\,T_\ell [\zeta](t)\,dt\,,
			\label{eq_lemme_bessel_sinus_1}\\
			\int_0^1\Psi_\ell (z t)\,\zeta(t)\,dt
			&=-\tfrac12
			\int_0^1\sin(2z t)\,T_\ell [\zeta](t)\,dt\,.
			\label{eq_lemme_bessel_sinus_2}
		\end{align}
		\item Its adjoint satisfies
		\[
		T_\ell ^*[\cos(2z x)]
		\;=\; 2\,\Phi_\ell (z x)-1\,,
		\qquad
		T_\ell ^*\bigl[-\tfrac12\sin(2z x)\bigr]
		\;=\; \Psi_\ell (z x)\,,
		\]
		with
		\(\ker(T_\ell ^*) = \mathrm{span}\{x^2,x^4,\dots,x^{2\ell}\}\,.\)
		\item \(T_\ell \) provides a Banach isomorphism
		\[
		L^2(0,1)\;\xrightarrow{\sim}\;
		\bigl(\ker T_\ell ^*\bigr)^{\perp}\,,
		\]
		whose inverse is
		\[
		B_\ell [f]
		\;=\;
		(-1)^{\ell+1}\,A_\ell \,A_{\ell -1}\cdots A_1[f]\,.
		\]
		Moreover,
		\[
		\Phi_\ell '(z x)
		\;=\;
		B_\ell [-\sin(2z x)]\,,
		\qquad
		\Psi_\ell '(z x)
		\;=\;
		B_\ell [-\cos(2z x)]\,.
		\]
	\end{enumerate}
\end{lemma}

\section{Green's identity via transformation operators}

In this section, we shall use the transformation operators $T_\ell$ to reformulate relation~\eqref{eq:zeta-sum} in a more convenient form. First, it is straightforward to verify that,
\[
J_\nu(xz)^2
=\frac{2}{\pi\,z\,x}\,\Phi_\ell (zx)\,,
\qquad
J_\nu(xz)\,Y_\nu(xz)
=-\frac{2}{\pi\,z\,x}\,\Psi_\ell (zx)\,.
\]
Substituting these into \eqref{eq:zeta-sum} gives
\begin{equation}\label{eq:PhiPsi-zeta}
J_\nu(z)\,\int_0^1\Psi_\ell (zx)\,\zeta(x)\,dx
\;+\;
Y_\nu(z)\,\int_0^1\Phi_\ell (zx)\,\zeta(x)\,dx
\;=\;0\,.
\end{equation}
Moreover, for functions \(\zeta\) orthogonal to the constants on \([0,1]\), i.e. satisfying 
$
\int_0^1 \zeta(x)\,dx = 0\,,
$
it follows from \eqref{eq:PhiPsi-zeta} that
\begin{equation}\label{eq:PhiPsi-zeta1}
	2\,J_\nu(z)\,\int_0^1\Psi_\ell (zx)\,\zeta(x)\,dx
	\;+\;
	Y_\nu(z)\,\int_0^1\bigl(2\,\Phi_\ell (zx)-1\bigr)\,\zeta(x)\,dx
	\;=\;0\,.
\end{equation}

\vspace{0.1cm}\noindent
We now invoke Lemma~\ref{lemme_bessel_sinus}(i) together with the identity
\[
Y_\nu(x)=(-1)^{\,\ell+1}J_{-\nu}(x)
\]
to deduce that, for such \(\zeta\), \eqref{eq:PhiPsi-zeta1} can be equivalently written as
\begin{equation} \label{eqintegrale}
J_\nu(z)\int_0^1 T_\ell [\zeta](x)\,\sin(2z x)\,dx
\;+\;
(-1)^{\,\ell}\,J_{-\nu}(z)\int_0^1 T_\ell [\zeta](x)\,\cos(2z x)\,dx
\;=\;0\,.
\end{equation}


\vspace{0.2cm}\noindent
We will now employ the half-integer Bessel functions and their associated polynomials \(P_\ell \) and \(Q_\ell \) (see \cite[10.1.19--20]{AS64}). For \(\ell=0,1,2,\dots\) and \(z\in\C\)\,,
\begin{align}
	J_{\ell +\frac12}(z)
	&=\sqrt{\frac{2}{\pi z}}\,
	\Bigl[P_\ell \!\bigl(\tfrac{1}{z}\bigr)\,\sin z
	\;-\;Q_{\ell -1}\!\bigl(\tfrac{1}{z}\bigr)\,\cos z\Bigr]\,, \label{besseldemientier1} \\ 
	J_{-\,\ell - \frac12}(z)
	&=(-1)^{\,\ell}\sqrt{\frac{2}{\pi z}}\,
	\Bigl[P_\ell \!\bigl(\tfrac{1}{z}\bigr)\,\cos z
	\;+\;Q_{\ell -1}\!\bigl(\tfrac{1}{z}\bigr)\,\sin z\Bigr]\,. \label{besseldemientier2}
\end{align}

\vspace{0.1cm}\noindent
The polynomials \(P_\ell (t)\) and \(Q_\ell (t)\), each of degree \(\ell \) in \(t\), satisfy the three-term recursion 
\begin{align}
	P_{\ell+1}(t)
	&=(2\ell+1)\,t\,P_\ell (t)\;-\;P_{\ell -1}(t)\,,
	\qquad \ell \ge 1, \\
	Q_{\ell+1}(t)
	&=(2\ell +3)\,t\,Q_\ell (t)\;-\;Q_{\ell -1}(t)\,,
	\qquad \ell \ge0,
\end{align}
with initial conditions
\[
P_0(t)=1\,,\quad P_1(t)=t\,,
\qquad
Q_{-1}(t)=0\,,\quad Q_0(t)=1\,.
\]
\medskip
\noindent
Note that \( P_\ell(t) \) and \( Q_\ell(t) \) are even polynomials if \( \ell \) is even, and odd polynomials if \( \ell \) is odd.

\vspace{0.1cm}\noindent
As a concrete example, the simplest half-integer Bessel functions are
\begin{equation}\label{bessel12}
	J_{\tfrac12}(z)
	=\sqrt{\frac{2}{\pi z}}\;\sin z\,,
	\qquad
	J_{-\tfrac12}(z)
	=\sqrt{\frac{2}{\pi z}}\;\cos z\,.
\end{equation}

\noindent
The next case corresponds to order \( \pm \tfrac{3}{2} \):
\begin{equation}\label{bessel32}
	J_{\tfrac{3}{2}}(z)
	=\sqrt{\frac{2}{\pi z}} \left( \frac{\sin z}{z} - \cos z \right)\,,
	\qquad
	J_{-\tfrac{3}{2}}(z)
	=\sqrt{\frac{2}{\pi z}} \left( -\frac{\cos z}{z} - \sin z \right)\,.
\end{equation}

\vspace{0.1cm}\noindent
The first few polynomials \(P_\ell (t)\) and \(Q_\ell (t)\) generated by the recursion  relations are:
\begin{equation}\label{polynome}
\begin{aligned}
	P_0(t)&=1\,,           & Q_{-1}(t)&=0\,,\\
	P_1(t)&=t\,,           & Q_0(t)   &=1\,,\\
	P_2(t)&=3t^2-1\,,      & Q_1(t)   &=3t\,,\\
	P_3(t)&=15t^3-6t\,,    & Q_2(t)   &=15t^2-1\,.
\end{aligned}
\end{equation}

\vspace{0.1cm}\noindent
From the foregoing we obtain the following:

\begin{prop}\label{lelemmedebase}
	Let \(\zeta\in L^2(0,1)\) be orthogonal to the constants and suppose further that for $\nu= \ell + \frac12$\,,
	\[
	\int_0^1 x\,\zeta(x)\,J_\nu\bigl(j_{\nu,n}\,x\bigr)^2\,dx
	\;=\;0\,,
	\qquad\text{for all }n\ge1\,.
	\]
	Then for every \(z\in\C\) and every integer \(\ell \ge 1\)\,,
	\begin{equation}\label{eqintegraledebase}
	\int_0^1 T_\ell [\zeta](x)\,
	\Bigl[
	P_\ell \!\bigl(\tfrac{1}{z}\bigr)\,\cos (z((2x-1))
	\;-\;
	Q_{\ell -1}\!\bigl(\tfrac1z\bigr)\,\sin (z(2x-1))
	\Bigr]
	\,dx
	\;=\;0\,.
	\end{equation}
\end{prop}

\begin{proof}
	The result follows immediately from \eqref{eqintegrale}, \eqref{besseldemientier1} and  \eqref{besseldemientier2}, by a straightforward application of elementary trigonometric identities.
\end{proof}

\begin{rem}
In the foregoing we have relied on the transformation operators \(S_\ell \) and \(T_\ell \) defined for \(\ell \ge 1\)\,.  For the case \(\ell=0\)\,, (i.e $\nu = \frac12$)\,, one obtains an analogous reduction by using the explicit formulas (\ref{bessel12}). It follows that
\[
\Phi_0(x)=\frac{1-\cos(2x)}{2}\,,
\qquad
\Psi_0(x)=\tfrac12\sin(2x)\,.
\]
Hence, under the hypotheses of Proposition~\ref{lelemmedebase}, it follows from \eqref{eq:PhiPsi-zeta1} that
\begin{equation}\label{lecasl=0}
\int_0^1 \zeta(x)\,\cos\bigl(z(2x-1)\bigr)\,dx = 0
\quad\text{for all }z\in\C\,.
\end{equation}
In other words, if one conventionally sets $T_0 = \mathrm{Id}$, then relation \eqref{eqintegraledebase} also extends to the case $\ell = 0\,$.
\end{rem}

\begin{rem}\label{ipp}
	Note that if one multiplies \eqref{eqintegraledebase} by \(z^\ell \) and then sets \(z=0\), it follows that for $\ell \geq 1$\,,
	\[
	\int_0^1 T_\ell [\zeta](x)\,dx = 0\,.
	\]
\end{rem}

\section{From Green's integral identities to differential relations}

Let \(\zeta \in L^2(0,1)\) be orthogonal to the constants and satisfy the hypotheses of Proposition~\ref{lelemmedebase}.  
Then, for each integer \(\ell \geq 0\), we shall show that the integral identity \eqref{eqintegraledebase} gives rise to a differential relation satisfied by the transformed function \(T_\ell[\zeta]\)\,.

\vspace{0.2cm}
\noindent
We introduce the reflection symmetry about the midpoint \( x = \tfrac{1}{2} \), defined by the involutive map
\[
\sigma(x) := 1 - x\,.
\]
This transformation naturally arises in several parts of the analysis, particularly when exploiting parity properties. Throughout the paper, we refer to this symmetry simply as \( \sigma \)\,.

\subsection{The case $\ell =0$}

By standard Fourier series theory, the relation (\ref{lecasl=0}) implies that \(\zeta\) is an odd function with respect to the midpoint \(x_0=\tfrac12\), i.e.
\[
\zeta(x) = -\,\zeta\bigl(\sigma(x)\bigr)
\quad\text{for almost every } x \in [0,1]\,.
\]

\subsection{The case $\ell=1$} 
According to \eqref{polynome}, and setting $g = T_{1}(\zeta) = S_{1}(\zeta)$, equation \eqref{eqintegraledebase} for $\ell = 1$ can be written as
\begin{equation}\label{verif}
	\int_{0}^{1} g(x)\,\Bigl[\tfrac{1}{z}\cos\bigl(z(2x-1)\bigr)
	-\sin\bigl(z(2x-1)\bigr)\Bigr]
	\,dx
	=0\,.
\end{equation}
Now, we define
\[
G(x) = \int_{0}^{x} g(t)\,\mathrm{d}t\,.
\]
By Remark~\ref{ipp}, \(G(1)=0\), so using integration by parts, we deduce
\begin{equation}
	\int_{0}^{1} \bigl[\,2G(x) - g(x)\bigr]\sin\bigl(z(2x-1)\bigr)\,\mathrm{d}x = 0.
\end{equation}
In other words, \(2G(x)-g(x)\) is an even function with respect to \(x=\tfrac12\)\,.  


\subsection{The case $\ell \geq 2$}

Before stating the result, we briefly recall the notion of symmetry for distributions.

\begin{defi}
	Let \( \sigma(x) := 1 - x \) denote the reflection with respect to the midpoint of the interval \( [0,1] \). This map induces a natural action on distributions \( T \in \mathcal{D}'(0,1) \) via pullback, defined by duality:
	\[
	\langle \sigma_* T, \varphi \rangle := \langle T, \varphi \circ \sigma \rangle, \qquad \text{for all } \varphi \in \mathcal{D}(0,1).
	\]
	We say that a distribution \( T \in \mathcal{D}'(0,1) \) is \emph{even} if \( \sigma_* T = T \)\,, and \emph{odd} if \( \sigma_* T = -T \)\,.
\end{defi}

\begin{rem}
	There is a classical characterization of even and odd distributions on the interval \( (0,1) \) using Fourier series adapted to the symmetry with respect to the midpoint \( x = \tfrac{1}{2} \)\,. 
	
	\vspace{0.1cm}\noindent
	Consider the orthonormal basis of \( L^2(0,1) \) formed by the functions
	\[
	\phi_n(x) := \sqrt{2} \cos\big(\pi n (2x - 1)\big)\,, \quad  n \geq 0 \mbox{ and } 
	\psi_n(x) := \sqrt{2} \sin\big(\pi n (2x - 1)\big)\,,  \quad n \geq 1\,.
	\]
	Each \( \phi_n \) is even with respect to \( x = \tfrac{1}{2} \)\,, while each \( \psi_n \) is odd.
	
	\vspace{0.1cm}\noindent
	Then, a distribution \( T \in \mathcal{D}'(0,1) \) is odd with respect to \( x = \tfrac{1}{2} \) if and only if
	\[
	\langle T, \phi_n \rangle = 0 \qquad \text{for all } n \geq 1\,.
	\]
	Similarly, \( T \) is even if and only if \( \langle T, \psi_n \rangle = 0 \) for all \( n \geq 1 \)\,.
	
	\vspace{0.2cm}\noindent
	Equivalently, a distribution \( T \in \mathcal{D}'(0,1) \) is odd with respect to the midpoint \( x = \tfrac{1}{2} \) if
	\[
	\langle T, \cos\big(z(2x - 1)\big) \rangle = 0 \qquad \text{for all } z \in \mathbb{C}\,,
	\]
	and it is even if
	\[
	\langle T, \sin\big(z(2x - 1)\big) \rangle = 0 \qquad \text{for all } z \in \mathbb{C}\,.
	\]
	
\end{rem}

\vspace{0.4cm}\noindent
We now introduce the following transformed polynomials:
\[
\widetilde{P}_\ell(z) := z^\ell P_\ell\left(\frac{1}{z}\right)\,, \qquad
\widetilde{Q}_{\ell-1}(z) := z^\ell Q_{\ell-1}\left(\frac{1}{z}\right)\,,
\]
where \( P_\ell(t) \) and \( Q_{\ell-1}(t) \) are as previously defined.

\medskip
\noindent
By construction, \( \widetilde{P}_\ell(z) \) is an even polynomial for all \( \ell \), while \( \widetilde{Q}_{\ell-1}(z) \) is always odd, regardless of the parity of \( \ell \).  We list below the first few examples:
\begin{align*}
	\widetilde{P}_1(z) &= 1\,, 
	&\qquad
	\widetilde{Q}_0(z) &= z\,, \\[0.3em]
	\widetilde{P}_2(z) &= 3 - z^2\,, 
	&\qquad
	\widetilde{Q}_1(z) &= 3z\,, \\[0.3em]
	\widetilde{P}_3(z) &= 15 - 6z^2\,, 
	&\qquad
	\widetilde{Q}_2(z) &= 15z - z^3\,.
\end{align*}

\vspace{0.1cm}\noindent
We also introduce the differential operator
\[
\widetilde{D} := \frac{1}{2i} \frac{d}{dx}\,,
\]
so that for any polynomial \( R(z) = \sum_k a_k z^k \), the corresponding operator \( R(\widetilde{D}) \) acts on a function \( f \) via
\[
R(\widetilde{D})f(x) = \sum_k a_k (\widetilde{D}^k f)(x)\,.
\]

\medskip
\noindent
We now proceed by multiplying the identity~\eqref{eqintegraledebase} by \( z^\ell \).  Under the assumptions of Lemma~\ref{lelemmedebase}, we obtain the following identity.

\begin{lemma}\label{symetrie}
	Let \( \zeta \in L^2(0,1) \) satisfy the hypotheses of Lemma~\ref{lelemmedebase}, and let \( g := T_\ell[\zeta] \)\,. Then, one has for all $z \in \C$\,,
	\[
	\int_0^1 g(x)\,\cdot\, \Bigl[
	\widetilde{P}_\ell(\widetilde{D}) + i\,\widetilde{Q}_{\ell-1}(\widetilde{D}) \Bigr] \cos(z(2x - 1))\,dx = 0\,.
	\]
\end{lemma}

\begin{rem}
	As a consistency check, let us consider the case \( \ell = 1 \)\,. In this case, one has
	\[
	\widetilde{P}_1(\widetilde{D}) + i\,\widetilde{Q}_0(\widetilde{D}) = 1 + \frac{1}{2} \frac{d}{dx}\,.
	\]
	Thus, we obtain:
	\[
	\int_0^1 g(x)\left(1 + \frac{1}{2} \frac{d}{dx}\right)\cos\big(z(2x - 1)\big)\,dx 
	= \int_0^1 g(x)\,\left[\cos\big(z(2x - 1)\big) - z \sin\big(z(2x - 1)\big)\right]\,dx = 0\,,
	\]
	as already observed in~\eqref{verif}.
\end{rem}

\medskip
\noindent
In particular, Lemma~\ref{symetrie} implies that the distribution
\[
\left( \widetilde{P}_\ell(\widetilde{D}) - i\,\widetilde{Q}_{\ell-1}(\widetilde{D}) \right) g(x)
\]
is odd with respect to the midpoint \( x = \tfrac{1}{2} \)\,. This follows from the fact that \( \widetilde{P}_\ell \) is an even polynomial, while \( \widetilde{Q}_{\ell-1} \) is odd.

\vspace{0.3cm}\noindent
We are now left with understanding more precisely the differential operator
\[
\left( \widetilde{P}_\ell(\widetilde{D}) - i\,\widetilde{Q}_{\ell-1}(\widetilde{D}) \right)\,.
\]
Using the recursion  relations satisfied by the polynomials \( P_\ell \) and \( Q_\ell \), and introducing
\[
\widetilde{A}_\ell(z) := \widetilde{P}_\ell(z) - i\,\widetilde{Q}_{\ell-1}(z)\,,
\]
we obtain the following lemma.

\begin{lemma}
	For every integer \( \ell \geq 1 \), the polynomials \( \widetilde{A}_\ell(z) \) satisfy the recursion  relation
	\[
	\widetilde{A}_{\ell+1}(z) - (2\ell+1) \widetilde{A}_\ell(z) + z^2 \widetilde{A}_{\ell-1}(z) = 0\,,
	\]
	with initial values
	\[
	\widetilde{A}_0(z) = 1\,, \qquad \widetilde{A}_1(z) = 1 - i z\,.
	\]
\end{lemma}

\vspace{0.3cm}\noindent
At this stage, the polynomials \( \widetilde{A}_\ell(z) \) have complex coefficients. To simplify the situation, we set  $A_\ell (t) = \tilde A_\ell (\frac{t}{2i})$. We obtain  the following result, which also holds for \( \ell = 0 \) and \( \ell = 1 \):

\begin{thm}\label{mainthm}
	Let \( \{A_\ell(t)\}_{\ell \in \mathbb{N}} \) be the sequence of polynomials defined recursively by
	\[
	A_0(t) = 1\,, \qquad A_1(t) = 1 - \frac{t}{2}\,,
	\]
	and for all \(\ell \geq 1 \)\,,
	\[
	A_{\ell+1}(t) = (2\ell + 1)\,A_\ell(t) + \frac{t^2}{4}\,A_{\ell-1}(t)\,.
	\]
	Suppose that \( \zeta \) satisfies the hypotheses of Lemma~\ref{lelemmedebase}. Then, in the sense of distributions, the function
	\[
	A_\ell\bigl(D\bigr)\bigl[T_\ell[\zeta]\bigr]\,,
	\]
	where \( D = \frac{d}{dx} \), is odd with respect to the midpoint \( x = \tfrac{1}{2} \)\,.
\end{thm}

\vspace{0.1cm}\noindent
\begin{rem}\label{rkcalculA_l}
The next polynomials in the sequence are given by
\[
A_2(t) \;=\; \tfrac{1}{4}t^2 - \tfrac{3}{2}t + 3\,,
\qquad
A_3(t) \;=\; -\tfrac{1}{8}t^3 + \tfrac{3}{2}t^2 - \tfrac{15}{2}t + 15\,.
\]
\end{rem}

\begin{rem}\label{rem:reverse-bessel}
The polynomials $A_\ell$ are well known. More precisely, if $\theta_\ell$ denotes the
\emph{reverse Bessel polynomial} (also called the \emph{reverse Bessel polynomial of degree $\ell$})\,,
then
\[
A_\ell(t)=\theta_\ell\!\left(-\frac{t}{2}\right),\qquad \ell\in\mathbb{N},
\]
see, e.g., \cite[ 18.34]{DLMF}.
In particular, $A_\ell$ admits the explicit expansion
\begin{equation}\label{eq:Aell-explicit}
A_\ell(t)
=\sum_{k=0}^{\ell}\frac{(\ell+k)!}{(\ell-k)!\,k!}\,\frac{(-1)^{\ell-k}}{2^{\ell}}\,t^{\ell-k}\,.
\end{equation}
Moreover, $A_\ell$ satisfies the following second order differential equation
\begin{equation}\label{eq:Aell-ode}
t\,A_\ell''(t)+\bigl(t-2\ell\bigr)A_\ell'(t)-\ell\,A_\ell(t)=0\,.
\end{equation}

\end{rem}

\section{Uniqueness result in the case \texorpdfstring{$(\ell_1,\ell_2)=(0,1)$}{(l1,l2)=(0,1)}}

In this section, we assume that \( \zeta \) satisfies the hypotheses of Proposition~\ref{lelemmedebase} for \( \ell = 0 \) and \( \ell = 1 \) and we aim to show that $\zeta$ vanishes almost everywhere on $(0,1)$\,. For \( \ell = 0 \)\,, this assumption simply implies that \( \zeta \) is odd with respect to the midpoint \( x = \tfrac{1}{2} \)\,. For \( \ell = 1 \)\,, it means by Theorem~\ref{mainthm} that the function \( A_1(D)\bigl[T_1[\zeta]\bigr] \) is odd in the sense of distributions.

\medskip
\noindent
The aim is to prove that \(y = \zeta'\) is a solution of a linear second-order differential equation on \((0,1)\) and satisfies the conditions
\[
y\!\left(\tfrac12\right) = y'\!\left(\tfrac12\right) = 0\,.
\]
It then follows that \(y\) vanishes identically on \((0,1)\)\,, and consequently \(\zeta\) vanishes almost everywhere on \((0,1)\)\,, since it is an odd function.

\medskip
\noindent
A straightforward computation yields
\begin{equation}\label{eq2A1(D)(T1)}
2 A_1(D)\bigl[T_1[\zeta]\bigr] (x)= 2 A_1(D)\bigl[S_1 [\zeta ]\bigr](x)
=  - \zeta'(x) +  \left( 2 - \frac{4}{x} \right) \zeta(x)  - 4(2x - 1) \int_x^1 \frac{\zeta(t)}{t^2} \, dt \,.
\end{equation}
We further compute, for all $x \in (0,1)$\,,
\[ G (x) := D^2 A_1(D)\bigl[T_1[\zeta]\bigr](x) = A_1(D)\bigl[ D^2 S_1[\zeta]\bigr](x) =  A_1(D)\left[ \zeta'' + \frac{4}{x} \zeta' \right](x)\]
where we used Lemma~\ref{rundel}~(iv) with \(l=1\). Setting $y := \zeta'$\,, we obtain for all \(x \in (0,1)\)\,,
\begin{equation}\label{expressionGl=1}
G(x) =  - \frac{1}{2}y''(x) + \left( 1 - \frac{2}{x} \right) y'(x) + \left( \frac{4}{x} + \frac{2}{x^2} \right)y(x) \, .
\end{equation}
Recalling that \(G\) is odd on $(0,1)$\,, it satisfies
\begin{equation}\label{equationGl=1}
G(x) + G(1-x) = 0
\end{equation}
in the sense of distributions. Since \(y\) is even, the identities~\eqref{expressionGl=1} and~\eqref{equationGl=1} imply that \(y\) satisfies a linear
second-order differential equation on \( (0,1) \). Moreover, since $\zeta$ is odd, $y' = \zeta''$ is also odd and $y' \left( \frac{1}{2} \right) = 0\,$. Finally,
evaluating~\eqref{eq2A1(D)(T1)} at $x = \frac{1}{2}$, we obtain
\[ y \!\left(\tfrac12\right) = \zeta' \!\left(\tfrac12\right) = - 2 A_1(D)\bigl[T_1[\zeta]\bigr] \!\left(\tfrac12\right) -6 \zeta \!\left(\tfrac12\right)  = 0\]
since both $\zeta$ and $A_1(D)\bigl[T_1[\zeta]\bigr]$ are odd functions on $(0,1)$.

\medskip
\noindent
The function \(y\) thus satisfies a linear
second-order differential equation on \( (0,1) \)\,, together with the conditions
$ y\!\left(\tfrac{1}{2}\right) = 0$ and $y'\!\left(\tfrac{1}{2}\right) = 0$\,.
By the Cauchy--Lipschitz theorem, we conclude that \( y \equiv 0 \)\,.
Therefore \( \zeta' = y = 0 \), so \( \zeta \) is constant.
Since \( \zeta \) is odd with respect to \( x = \tfrac{1}{2} \)\,,
this constant must vanish, and $\zeta \equiv 0$\,. This completes the proof of the following completeness result.

\begin{thm}
	Assume that \( \zeta \in L^2(0,1) \) satisfies the hypotheses of Proposition~\ref{lelemmedebase} for \( \ell = 0 \) and \( \ell = 1 \)\,. Then \( \zeta = 0 \) almost everywhere on \( (0,1) \)\,.
\end{thm}

\noindent
In particular, this implies that the differential spectral map associated with the pair \( (\ell_1,\ell_2) = (0,1) \) is injective.

\section{Uniqueness result in the case \texorpdfstring{$(\ell_1,\ell_2)=(0,2)$}{(l1,l2)=(0,2)}}

Throughout this subsection, we assume that \( \zeta \) satisfies the hypotheses of Proposition~\ref{lelemmedebase} for both \( \ell = 0 \) and \( \ell = 2 \) and we aim to show that $\zeta$ vanishes almost everywhere on $(0,1)$\,. 
As before, the condition for \( \ell = 0 \) implies that \( \zeta \) is odd with respect to the midpoint \( x = \tfrac{1}{2} \) and for \( \ell = 2 \)\,, it means by Theorem~\ref{mainthm} that the function \( A_2(D)\bigl[T_2[\zeta]\bigr] \) is odd in the sense of distributions.

\medskip
\noindent
In order to prove that $\zeta$ vanishes almost everywhere on $(0,1)$, we follow the same strategy as in the previous Subsection. The aim is to prove that \(y = \zeta'\) is a solution of a linear fourth-order differential equation on \((0,1)\) and satisfies the conditions
\[
y\!\left(\tfrac12\right) = y'\!\left(\tfrac12\right) = y''\!\left(\tfrac12\right) = y^{(3)}\!\left(\tfrac12\right) =0\,.
\]
It then follows that \(y\) vanishes identically on \((0,1)\)\,, and consequently \(\zeta\) vanishes almost everywhere on \((0,1)\)\,, since it is an odd function.

\medskip
\noindent
Let us set \( f := S_1[\zeta] \) and \(g = -S_2[f]\), so that \( g = T_2[\zeta] \)\,. We now differentiate the expression \( A_2(D)\bigl[g \bigr] \) four times. Since the application of an even number of derivatives preserves parity in the distributional sense, and since the operator \( A_2(D) \) commutes with differentiation, it follows that
\[
A_2(D)\bigl[ f^{(4)}(x) + \tfrac{8}{x} f^{(3)}(x) \bigr]
\]
is odd, where we have used Lemma~\ref{rundel}~(iv) with \( \ell = 2 \)\,. Using again Lemma~\ref{rundel}~(iv), now with \( \ell = 1 \)\,, and substituting into the above expression, we obtain that the function 
$A_2(D)\bigl[F\bigr]$ is odd in the sense of distributions, where for all $x\in (0,1)$\,,
\[
F(x) := \zeta^{(4)}(x) + \frac{12}{x}\,\zeta^{(3)}(x) + \frac{24}{x^2} \zeta''(x)
- \frac{24}{x^3} \zeta'(x)\,.
\]
As in the cases \( \ell = 0 \) and \( \ell = 1 \), we now define \( y(x) := \zeta'(x) \)\,. Since \( \zeta \) is odd with respect to the midpoint \( x = \tfrac{1}{2} \), it follows that \( y \) is an even function. Applying the differential operator \(4  A_2(D) = D^2 - 6D + 12 \) to the function $F$, we define
\[
G(x) := 4  A_2(D)[F(x)] = \left(\frac{d^2}{dx^2} - 6\frac{d}{dx} + 12\right) 
\left( y^{(3)} + \frac{12}{x}\,y'' + \frac{24}{x^2} y' - \frac{24}{x^3} y \right)
\]
and we find:
\begin{align}
	G(x) = &\left( -\frac{288}{x^5} - \frac{432}{x^4} - \frac{288}{x^3} \right) y(x)
	+ \left( \frac{288}{x^4} + \frac{432}{x^3} + \frac{288}{x^2} \right) y'(x)   \nonumber\\
	&+ \left( -\frac{96}{x^3} - \frac{72}{x^2} + \frac{144}{x} \right) y''(x)
	+ \left( 12 - \frac{72}{x} \right) y^{(3)}(x)   \nonumber\\
	&+ \left( -6 + \frac{12}{x} \right) y^{(4)}(x)
	+ y^{(5)}(x)\,.\nonumber
\end{align}
\noindent
By the symmetry result, we have that $G(x) + G(1 - x) = 0$  in the sense of distributions. This implies that the function \( y \) satisfies a linear differential equation of order four, since the fifth derivative \( y^{(5)} \) is odd with respect to the midpoint \( x = \tfrac{1}{2} \)\,. 

\medskip
\noindent
We now prove that $ y\!\left(\tfrac{1}{2}\right) = y'\!\left(\tfrac{1}{2}\right) =  y''\!\left(\tfrac{1}{2}\right) =   y^{(3)} \!\left(\tfrac{1}{2}\right) = 0 $\,. First, since $y$ is even, $y'$ and $y^{(3)}$ are both odd and $y'\!\left(\tfrac{1}{2}\right) =   y^{(3)} \!\left(\tfrac{1}{2}\right) = 0$\,. Recall that (see (3.9) in \cite{RuSa01}), for all $x \in (0,1)$\,,
\[ g (x) = T_2[\zeta](x) = - \zeta(x) - 12x \int_x^1 \frac{\zeta(t)}{t^2} \, dt + 24 x^3 \int_x^1 \frac{\zeta(t)}{t^4} \, dt \,.\]
Applying the differential operator $4A_2(D)=D^{2}-6D+12$ to $g$, we obtain for all $x \in (0,1)$\,,
\begin{align}
	4A_2(D)[g](x) =  &- \zeta''(x) + \left( 6 - \frac{12}{x} \right) \zeta'(x) + \left( - \frac{48}{x^2} + \frac{72}{x} - 12 \right) \zeta(x)   \nonumber\\
	&+ 72(1 - 2x) \int_x^1 \frac{\zeta(t)}{t^2} \, dt + 144 x(x-1)(2x - 1) \int_x^1 \frac{\zeta(t)}{t^4} \, dt \,.\nonumber
\end{align}
Evaluating this expression at $x=\tfrac12$ and using that $\zeta$ is odd and $y = \zeta'$, we obtain
\[ 4A_2(D)\bigl[g\bigr]\!\left(\tfrac12\right)
=
- \zeta''\!\left(\tfrac12\right) -18 \zeta' \!\left(\tfrac12\right) -60 \zeta \!\left(\tfrac12\right)
= -18\,y\!\left(\tfrac12\right)\,. \]
Since $4A_2(D)\bigl[g\bigr](x)$ is odd, we therefore conclude that $y\!\left(\tfrac12\right)=0$\,.

\medskip
\noindent
Proceeding in the same way, we compute
\begin{align}
	4D^2 A_2(D)[g](x) =  &- \zeta^{(4)}(x) + \left( 6 - \frac{12}{x} \right) \zeta^{(3)}(x) + \left( - \frac{24}{x^2} + \frac{72}{x} -12 \right) \zeta''(x) \nonumber \\
	& + \left( \frac{24}{x^3} + \frac{216}{x^2} - \frac{144}{x} \right) \zeta'(x) + \left( \frac{288}{x^3} - \frac{576}{x^2}  \right) \zeta(x)   \nonumber\\
	&+ 864 (2x - 1 ) \int_x^1 \frac{\zeta(t)}{t^4} \, dt \,.\nonumber
\end{align}
Recalling that $\zeta$ is odd, $y = \zeta'$ and $y\!\left(\tfrac12\right)=0$\,, we evaluate at $x=\tfrac12$ and obtain 
\[
 4\,D^{2}A_{2}(D)\bigl[g\bigr]\!\left(\tfrac12\right)
=
-\,y^{(3)}\!\left(\tfrac12\right)
-18\,y''\!\left(\tfrac12\right)
+36\,y'\!\left(\tfrac12\right)
+768\,y\!\left(\tfrac12\right)
= -18\,y''\!\left(\tfrac12\right)\,.
\]
Since $4\,D^{2}A_{2}(D)\bigl[g\bigr]$ is also an odd function, we conclude, as above, that $ y''\!\left(\tfrac12\right)=0$\,.

\medskip\noindent
In conclusion, $y$ satisfies a fourth--order linear differential equation with
the initial conditions
\[ 
y\!\left(\tfrac12\right)=0,\qquad
y'\!\left(\tfrac12\right)=0,\qquad
y''\!\left(\tfrac12\right)=0,\qquad
y^{(3)}\!\left(\tfrac12\right)=0\,.
\]
The Cauchy--Lipschitz theorem then implies that $y\equiv 0$.
Since $y=\zeta'$ and $\zeta$ is odd, this in turn forces $\zeta \equiv 0$\,. We thus obtain a completeness result in the case \( (\ell_1,\ell_2)=(0,2) \), in full analogy with the case \( (\ell_1,\ell_2)=(0,1) \):

\begin{thm}
	Assume that \( \zeta \in L^2(0,1) \) satisfies the hypotheses of Proposition~\ref{lelemmedebase} for \( \ell = 0 \) and \( \ell = 2 \). Then \( \zeta = 0 \) almost everywhere on \( (0,1) \).
\end{thm}

\noindent
In particular, this implies that the differential of the spectral map associated with the pair \( (\ell_1,\ell_2)=(0,2) \) is injective.



\section{Uniqueness result in the case \texorpdfstring{$(\ell_1,\ell_2)=(1,2)$}{(l1,l2)=(1,2)}}

In this section, we assume that $\zeta$ satisfies the hypotheses of Proposition~\ref{lelemmedebase} for both $\ell = 1$ and $\ell = 2\,$. 
Our goal is to prove that $\zeta$ vanishes identically on $(0,1)\,$.

\vspace{0.2cm}\noindent
We first consider the case $\ell=1\,$. 
By Theorem~\ref{mainthm}, the distribution
$
A_1(D)\bigl[S_1[\zeta]\bigr]
$
is odd, where $A_1(t)=1-\frac{t}{2}$ and $D=\frac{d}{dx}\,$.
Thus, if we denote $f := S_1[\zeta]\,$, this condition is equivalent to requiring that
$
f' - 2f
$
be odd (in the sense of distributions). 
Decomposing $f$ into its even and odd parts,
\begin{equation}\label{decomparity}
f = f_e + f_o\,,
\end{equation}
where $f_e$ is even and $f_o$ is odd with respect to $x=\tfrac12$\,, we immediately get 
\begin{equation}\label{linkparity}
f_e = \frac{1}{2} f_o'\,,
\end{equation} 
and therefore
\begin{equation}\label{formf}
f =  \frac{1}{2} f_o'+ f_o \,.
\end{equation}

\medskip
\noindent
We now exploit the case $\ell=2\,$. 
Since $\zeta$ satisfies the assumptions of Proposition~\ref{lelemmedebase} with $\ell=2$, 
Theorem~\ref{mainthm} yields that
$
A_2(D)\bigl[T_2[\zeta]\bigr]
$
is odd in the sense of distributions, where 
\[
A_2(t)=\tfrac14 t^2-\tfrac32 t+3\,.
\]
We define
\[
g := S_2[f]
= f(x)-8x^{3}\int_x^1 \frac{f(t)}{t^{4}}\,dt\,,
\]
so that $g=-T_2[\zeta]\,$. A straightforward computation yields
\begin{equation}\label{parity1}
4A_2(D)[g](x)
=
f''(x)
+\Bigl(\frac{8}{x}-6\Bigr)f'(x)
+\Bigl(\frac{16}{x^2}-\frac{48}{x}+12\Bigr)f(x)
-48x(1-x)(1-2x)\int_x^1 \frac{f(t)}{t^4}\,dt\,.
\end{equation}
Since $4A_2(D)[g]$ is odd with respect to $x=\tfrac12$\,, evaluating \eqref{parity1} at $x=\tfrac12$
yields
\[
f''\!\left(\tfrac12\right)
+10\,f'\!\left(\tfrac12\right)
-20\,f\!\left(\tfrac12\right)
=0\,.
\]
Using the decomposition $f=\tfrac12 f_o'+f_o$, where $f_o$ is odd, we obtain
\begin{equation}\label{eq:fo-third-derivative}
	f_o^{(3)}\!\left(\tfrac12\right)=0\,.
\end{equation}
Similarly, a straightforward computation gives
\begin{equation}\label{eq:D2A2g}
	\begin{aligned}
		D^{2}\bigl(4A_2(D)[g]\bigr)(x)
		&=
		f^{(4)}(x)
		+\Bigl(\frac{8}{x}-6\Bigr)f^{(3)}(x)
		+\Bigl(12-\frac{48}{x}\Bigr)f''(x)
		\\[0.3em]
		&\quad
		+\Bigl(\frac{96}{x}-\frac{48}{x^{2}}\Bigr)f'(x)
		+\Bigl(\frac{192}{x^{2}}-\frac{96}{x^{3}}\Bigr)f(x)
		\\[0.3em]
		&\quad
		+288(1-2x)\int_x^1 \frac{f(t)}{t^{4}}\,dt\,.
	\end{aligned}
\end{equation}
Since $D^{2}\bigl(4A_2(D)[g]\bigr)$ is odd with respect to $x=\tfrac12$, evaluating at $x=\tfrac12$
gives
\begin{equation}\label{eq:fo-fifth-derivative}
	f_o^{(5)}\!\left(\tfrac12\right)
	=
	64\,f_o^{(3)}\!\left(\tfrac12\right)=0\,.
\end{equation}
Finally, a similar computation yields
\begin{equation}\label{eq:G-definition}
	\begin{aligned}
		G(x):= D^{4}\bigl(4A_2(D)[g]\bigr)(x)
		&=
		f^{(6)}(x)
		+\Bigl(\frac{8}{x}-6\Bigr)f^{(5)}(x)
		+\Bigl(12-\frac{48}{x}-\frac{16}{x^{2}}\Bigr)f^{(4)}(x)
		\\[0.3em]
		&\quad
		+\Bigl(\frac{96}{x}+\frac{48}{x^{2}}+\frac{16}{x^{3}}\Bigr)f^{(3)}(x)\,.
	\end{aligned}
\end{equation}
Replacing, $f =  \frac{1}{2} f_o'+ f_o$, we get
\begin{equation}\label{eq:G-fo-expanded}
	\begin{aligned}
		G(x)
		&=
		\frac12\,f_o^{(7)}(x)
		+\Bigl(\frac{4}{x}-2\Bigr)f_o^{(6)}(x)
		-\Bigl(\frac{16}{x}+\frac{8}{x^{2}}\Bigr)f_o^{(5)}(x)
		\\[0.3em]
		&\quad
		+\Bigl(12+\frac{8}{x^{2}}+\frac{8}{x^{3}}\Bigr)f_o^{(4)}(x)
		+\Bigl(\frac{96}{x}+\frac{48}{x^{2}}+\frac{16}{x^{3}}\Bigr)f_o^{(3)}(x)\,.
	\end{aligned}
\end{equation}

\vspace{0.2cm}\noindent
We now use the symmetry condition $G(x)+G(1-x)=0$. 
This yields the following differential equation.
\begin{equation}\label{eq:symmetry-fo}
	\begin{aligned}
		0={}&
		f_o^{(7)}
		+4\left(\frac1x-\frac1{1-x}\right) f_o^{(6)}
		\\
		&-\left[16\left(\frac1x+\frac1{1-x}\right)
		+8\left(\frac1{x^2}+\frac1{(1-x)^2}\right)\right] f_o^{(5)}
		\\
		&+8\left[\left(\frac1{x^2}-\frac1{(1-x)^2}\right)
		+\left(\frac1{x^3}-\frac1{(1-x)^3}\right)\right] f_o^{(4)}
		\\
		&+\left[96\left(\frac1x+\frac1{1-x}\right)
		+48\left(\frac1{x^2}+\frac1{(1-x)^2}\right)
		+16\left(\frac1{x^3}+\frac1{(1-x)^3}\right)\right] f_o^{(3)}\,.
	\end{aligned}
\end{equation}
Finally, setting
\[
y(x)=f_o^{(3)}(x)\,,
\]
we obtain a fourth--order differential equation for $y$\,. 
The function $y$ is even with respect to $x=\tfrac12$, and from the previous identities we have
\[
y^{(k)}\!\left(\tfrac12\right)=0\,, \qquad k=0,1,2,3\,.
\]
By the Cauchy--Lipschitz theorem, it follows that
$y\equiv 0$. Hence $f_o$ is an odd polynomial of degree at most two, 
and therefore necessarily of the form
\begin{equation}\label{eq:fo-linear}
	f_o(x)=a\,(x-\tfrac12)\,.
\end{equation}
Using again $f=\tfrac12 f_o'+f_o\,$, we obtain
\begin{equation}\label{eq:f-linear}
	f(x)=\frac{a}{2}+a\,(x-\tfrac12)=a\,x\,.
\end{equation}
Since \( f = S_1[\zeta] \), applying the left inverse \( A_1 \) of \( S_1 \) given in Lemma~\ref{rundel}(iii) yields
\begin{equation}\label{eq:zeta-zero}
	\zeta(x)
	=
	A_1[f](x)
	=
	f(x)-\frac{4}{x^{3}}\int_{0}^{x} t^{2}\,f(t)\,dt
	=
	0\,.
\end{equation}

\medskip\noindent
We have thus proved the following theorem.

\begin{thm}
	Assume that $\zeta \in L^2(0,1)$ satisfies the hypotheses of 
	Proposition~\ref{lelemmedebase} for $\ell=1$ and $\ell=2$.
	Then $\zeta=0$ almost everywhere on $(0,1)$\,.
\end{thm}

\noindent
In particular, the differential of the spectral map associated with the pair 
$(\ell_1,\ell_2)=(1,2)$ is injective.


\section{Uniqueness result in the case \texorpdfstring{$(\ell_1,\ell_2)=(0,3)$}{(l1,l2)=(0,3)}}

Throughout this section, we assume that \( \zeta \) satisfies the hypotheses of Proposition~\ref{lelemmedebase} for both \( \ell = 0 \) and \( \ell = 3 \)\,, and we aim to show that $\zeta$ vanishes almost everywhere on $(0,1)$\,. 
As before, the condition for \( \ell = 0 \) implies that \( \zeta \) is odd with respect to the midpoint \( x = \tfrac{1}{2} \)\,. For \( \ell = 3 \)\,, Theorem~\ref{mainthm} implies that the function \( A_3(D)\bigl[T_3[\zeta]\bigr] \) is odd in the sense of distributions.

\medskip
\noindent
In order to prove that $\zeta$ vanishes almost everywhere on $(0,1)$, we follow the same strategy as in the previous section. The only step where we rely on computer assistance is Step~3, where we provide a computer-assisted proof.

\begin{itemize}
\item \textbf{Step~1.}
We show that
\[
\Bigl( \forall\, k \in \{0,1,2\},\ 
D^{2k} A_3(D)\bigl[T_3[\zeta]\bigr]\!\left(\tfrac12\right) = 0 \Bigr)
\quad \Longrightarrow \quad (S)\,,
\]
where $(S)$ denotes a triangular linear system involving the even-order derivatives of
\( y = \zeta' \) at \( x = \tfrac12 \)\,, uniquely determined by the single value \( y(\tfrac12) \)\,.

\item \textbf{Step~2.}
We show that \( y \) satisfies a linear differential equation of order~\(8\) on the interval \( (0,1) \).
The associated indicial equation at \( x = 0 \) has the characteristic roots
\[
\rho = -2,\; 1,\; 3,\; 5,\; 6,\; 7,\;
-2 \pm 2 i \sqrt{11}\,.
\]

\item \textbf{Step~3.}
A numerical investigation shows that the unique solution of this differential equation
satisfying \( y(\tfrac12) = 1 \) and whose initial data at \( x = \tfrac12 \) are prescribed by the system $(S)$
exhibits a pronounced blow-up as \( x \to 0^+ \) and \( x \to 1^- \).
The numerical solution \( y(x) \) displays oscillatory growth compatible with the Frobenius exponents
\(\alpha = -2 \pm 2 i \sqrt{11}\)\,.
As a consequence, the corresponding function \( \zeta \) fails to be square-integrable on \( (0,1) \)\,.

\item \textbf{Step~4.}
We infer that all derivatives of \( y \) at \( x = \tfrac12 \) must vanish.
It then follows that \( y \equiv 0 \) on \( (0,1) \), and consequently
\( \zeta = 0 \) almost everywhere on \( (0,1) \), since \( \zeta \) is odd with respect to
the midpoint \( x = \tfrac12 \)\,.
\end{itemize}


\medskip
\noindent
\textbf{Step 1.a: for all odd $k \in \{0,\dots 7\}$, i.e., $k \in \{1,3,5,7\}$\, , $y^{(k)}\!\left(\tfrac12\right) = 0$\,.} This is an immediate consequence of the fact that $y^{(k)}$ is odd whenever $k$ is odd.

\medskip
\noindent
\textbf{Step~1.b.}
We derive a triangular linear system \((S)\) for the even-order jets of \(y\) at \(x=\tfrac12\)\,.
We follow the strategy used in cases $(0,1)$ and $(0,2)$ but as we will see, due to the increasing degree of the polynomial $A_{\ell}$\,, we need to work a little more.

\medskip
\noindent
\textbf{Step 1.b.i: using $A_3(D)\bigl[T_3[\zeta]\bigr] \!\left(\tfrac12\right) = 0$}. Recall that (see (3.9) in \cite{RuSa01}), for all $x \in (0,1)$\,,
\[ h (x) = T_3[\zeta](x) = \zeta(x) - 24x \int_x^1 \frac{\zeta(t)}{t^2} \, dt + 120 x^3 \int_x^1 \frac{\zeta(t)}{t^4} \, dt -120 x^5 \int_x^1 \frac{\zeta(t)}{t^6} \, dt \,.\]
Applying the differential operator $8A_3(D)=-D^3 + 12 D^2 - 60D + 120$ to $h$, we obtain for all $x \in (0,1)$\,,
\begin{align}
	8A_3(D)[h](x) =  &- \zeta'''(x) + \left( 12 - \frac{24}{x} \right) \zeta''(x)  + \left( -60 + \frac{288}{x} - \frac{216}{x^2} \right) \zeta'(x)   \nonumber\\
	&  + \left( 120 - \frac{1440}{x} + \frac{2880}{x^2}  - \frac{1200}{x^3} \right) \zeta(x) + 1440(-2x + 1 )  \int_x^1 \frac{\zeta(t)}{t^2} \, dt
	   \nonumber\\
	&+ 720( 20x^3 - 30 x^2 + 12 x - 1)   \int_x^1 \frac{\zeta(t)}{t^4} \, dt  \nonumber\\
	&+ 7200( -2x^5 +5x^4 - 4 x^3 + x^2) \int_x^1 \frac{\zeta(t)}{t^6} \, dt \,.\nonumber
\end{align}
Evaluating this expression at $x=\tfrac12$ and using that $\zeta$ is odd and $y = \zeta'$, we obtain
\[ 8A_3(D)\bigl[h\bigr]\!\left(\tfrac12\right)
= - \zeta'''\!\left(\tfrac12\right)
- 36 \zeta''\!\left(\tfrac12\right) - 348 \zeta' \!\left(\tfrac12\right) -840 \zeta \!\left(\tfrac12\right)
= - y''\!\left(\tfrac12\right) - 348 \,y\!\left(\tfrac12\right)\,. \]
Since $8A_3(D)[h]$ is odd, we therefore conclude that 
\[  - y''\!\left(\tfrac12\right) - 348 \,y\!\left(\tfrac12\right)=0 \,. \]

\medskip
\noindent
\textbf{Step 1.b.ii: using $D^2 A_3(D)\bigl[T_3[\zeta]\bigr] \!\left(\tfrac12\right) = 0$.} After some calculations we obtain :
\[ 8D^2A_3(D)\bigl[h\bigr]\!\left(\tfrac12\right)
= - \zeta^{(5)}\!\left(\tfrac12\right)
- 36 \zeta^{(4)}\!\left(\tfrac12\right) -156 \zeta'''\!\left(\tfrac12\right)
+3384 \zeta''\!\left(\tfrac12\right) +18432 \zeta ' \!\left(\tfrac12\right)
\,. \]
Since $8 D^2 A_3(D)[h]$ is odd, we therefore conclude that 
\[  - y^{(4)}\!\left(\tfrac12\right) -156 y''\!\left(\tfrac12\right) +18432 \,y\!\left(\tfrac12\right) = 0 \, . \]

\medskip
\noindent
\textbf{Step 1.b.iii: using $D^4 A_3(D)\bigl[T_3[\zeta]\bigr] \!\left(\tfrac12\right) = 0$.} After some calculations we obtain :
\[ 8D^4A_3(D)\bigl[h\bigr]\!\left(\tfrac12\right)
= - \zeta^{(7)}\!\left(\tfrac12\right)
- 36 \zeta^{(6)}\!\left(\tfrac12\right) +36 \zeta^{(5)}\!\left(\tfrac12\right)
+6072 \zeta^{(4)}\!\left(\tfrac12\right) -18432 \zeta'''\!\left(\tfrac12\right)
-294912 \zeta''\!\left(\tfrac12\right) +589824 \zeta ' \!\left(\tfrac12\right)
\,. \]
Since $8 D^4 A_3(D)[h]$ is odd, we therefore conclude that 
\[ - y^{(6)}\!\left(\tfrac12\right) +36 y^{(4)}\!\left(\tfrac12\right)
-18432 y''\!\left(\tfrac12\right)+589824 y \!\left(\tfrac12\right) = 0 \, . \]

\medskip
\noindent

\vspace{0.2cm}\noindent
Thus, in contrast with the case $(0,2)$, where for \( k \in \{0,1\} \) one has
\[
D^{2k} A_2(D)\bigl[T_2[\zeta]\bigr]\!\left(\tfrac12\right) = 0
\quad \Longrightarrow \quad
y^{(2k)}\!\left(\tfrac12\right) = 0\,,
\]
we obtain here
\[
\Bigl( \forall\, k \in \{0,1,2\}\,,\ 
D^{2k} A_3(D)\bigl[T_3[\zeta]\bigr]\!\left(\tfrac12\right) = 0 \Bigr)
\quad \Longrightarrow \quad (S)\,,
\]
where $(S)$ denotes the following triangular linear system:
\[
\left\{
\begin{alignedat}{9}
 &&&&&& & & & & & y''\!\left(\tfrac12\right) &{}+{}& 348 &\,y\!\left(\tfrac12\right) &= 0,\\
&&&&& & & & y^{(4)}\!\left(\tfrac12\right) &{}+{}& 156 &\,y''\!\left(\tfrac12\right) &{}-{}& 18432 &\,y\!\left(\tfrac12\right) &= 0,\\
&&&& & y^{(6)}\!\left(\tfrac12\right) &{}-{}& 36 &\,y^{(4)}\!\left(\tfrac12\right) &{}+{}& 18432 &\,y''\!\left(\tfrac12\right) &{}-{}& 589824 &\,y\!\left(\tfrac12\right) &= 0\,.
\end{alignedat}
\right.
\]
This system is uniquely determined by the single value \( y(\tfrac12) \)\,.

\vspace{0.2cm}\noindent

\medskip
\noindent
\textbf{Step 2 : $y$ is a solution of a linear $8^{\text{th}}$-order differential equation.} Let us set \( f := S_1[\zeta] \)\,, \(g := S_2[f]\)\,, and \(h := S_3[g]\), so that \( h = T_3[\zeta] = S_3 S_2 S_1 [\zeta] \)\,. Using Lemma~\ref{rundel}~(iv) with \( \ell = 3 \), \( \ell = 2\) and \( \ell = 1\)\,, we have
\[ D^6[h] = D^6[g] + \frac{12}{x} D^5[g] \quad ; \quad D^4[g] = D^4[f] + \frac{8}{x} D^3[f] \quad ; \quad D^2[f] = D^2[\zeta] + \frac{4}{x} D[\zeta] .\]
Then, we can write $D^6[h]$ as
\begin{equation}\label{exprD6hl=3}
 D^6[h] = D^6[\zeta] + \sum_{k=1}^5 a_k \!\left(\tfrac1x\right) D^k [\zeta]
 \end{equation}
where for all $k \in \{1, \dots ,5\}\,$, $a_k$ is a polynomial function. We now differentiate the expression \( A_3(D)\bigl[h \bigr] \) six times. Since applying an even number of derivatives preserves parity in the distributional sense, and since the operator \( A_3(D) \) commutes with differentiation, it follows that the function
\[G := 8 D^6 A_3(D)\bigl[h \bigr]= 8 A_3(D)\bigl[ D^6 [h] \bigr]\]
is odd. Using the expression of $A_3(D)$ given in Remark~\ref{rkcalculA_l} and the form of $D^6[h]$ given in~\eqref{exprD6hl=3}, we obtain that the odd function $G$ can be written as
\[ G (x) = (-D^3 + 12 D^2 - 60D + 120)\bigl[  D^6[\zeta]\bigr](x) + \sum_{k=1}^5 a_k \!\left(\tfrac1x\right) D^k [\zeta] (x) = - D_9[\zeta](x) + \sum_{k=1}^8 b_k \!\left(\tfrac1x\right) D^k [\zeta](x)\,,
 \]
where for all $k \in \{1,\dots,8\}$, $b_k$ is a polynomial function. Thus, since $G$ is odd on $(0,1)$, it satisfies $G(x) + G(1-x) = 0$ in the sense of distributions. It follows that $y = \zeta' = D[\zeta]$ satisfies the following  linear $8^{\text{th}}$-order differential equation on \((0,1)\):

\begin{lemma}\label{edo03bis}
	Assume that \( \zeta \in L^2(0,1) \) satisfies the hypotheses of Lemma~\ref{lelemmedebase} 
	for \( \ell = 0 \) and \( \ell = 3 \)\,, and set \( y(x) := \zeta'(x) \)\,. 
	Then \( y \) satisfies the following differential equation
	\[
	-\,y^{(8)}(x)
	+12\left( \frac{1}{1-x} - \frac{1}{x} \right) y^{(7)}(x)
	\]
	\[
	+\left[
	-60
	+144\left( \frac{1}{x} + \frac{1}{1-x} \right)
	-36\left( \frac{1}{x^2} + \frac{1}{(1-x)^2} \right)
	\right] y^{(6)}(x)
	\]
	\[
	+\left[
	-720\left( \frac{1}{x} - \frac{1}{1-x} \right)
	+576\left( \frac{1}{x^2} - \frac{1}{(1-x)^2} \right)
	+336\left( \frac{1}{x^3} - \frac{1}{(1-x)^3} \right)
	\right] y^{(5)}(x)
	\]
	\[
	+\left[
	-720\left( \frac{1}{x^4} + \frac{1}{(1-x)^4} \right)
	-2880\left( \frac{1}{x^3} + \frac{1}{(1-x)^3} \right)
	-3600\left( \frac{1}{x^2} + \frac{1}{(1-x)^2} \right)
	+1440\left( \frac{1}{x} + \frac{1}{1-x} \right)
	\right] y^{(4)}(x)
	\]
	\[
	+\left[
	-2880\left( \frac{1}{x^5} - \frac{1}{(1-x)^5} \right)
	+7200\left( \frac{1}{x^3} - \frac{1}{(1-x)^3} \right)
	+8640\left( \frac{1}{x^2} - \frac{1}{(1-x)^2} \right)
	\right] y^{(3)}(x)
	\]
	\[
	+\left[
	23040\left( \frac{1}{x^6} + \frac{1}{(1-x)^6} \right)
	+34560\left( \frac{1}{x^5} + \frac{1}{(1-x)^5} \right)
	+21600\left( \frac{1}{x^4} + \frac{1}{(1-x)^4} \right)
	+2880\left( \frac{1}{x^3} + \frac{1}{(1-x)^3} \right)
	\right] y''(x)
	\]
	\[
	-\left[
	60480\left( \frac{1}{x^7} - \frac{1}{(1-x)^7} \right)
	+103680\left( \frac{1}{x^6} - \frac{1}{(1-x)^6} \right)
	+86400\left( \frac{1}{x^5} - \frac{1}{(1-x)^5} \right)
	+34560\left( \frac{1}{x^4} - \frac{1}{(1-x)^4} \right)
	\right] y'(x)
	\]
	\[
	+\left[
	60480\left( \frac{1}{x^8} + \frac{1}{(1-x)^8} \right)
	+103680\left( \frac{1}{x^7} + \frac{1}{(1-x)^7} \right)
	+86400\left( \frac{1}{x^6} + \frac{1}{(1-x)^6} \right)
	+34560\left( \frac{1}{x^5} + \frac{1}{(1-x)^5} \right)
	\right] y(x)
	\;=\;0\,,
	\]
	in the sense of distributions on \( (0,1) \)\,.
\end{lemma}

\medskip\noindent
The indicial equation at \(x=0\) associated with the differential equation of Lemma~\ref{edo03bis} takes the form
	\[
-\;(\rho-7)(\rho-6)(\rho-5)(\rho-3)(\rho-1)(\rho+2)
\left(\rho^2+4\rho+48\right)=0\,.
\]
Consequently, the corresponding indicial roots are
\[
\rho \;=\;-2,\;1,\;3,\;5,\;6,\;7,\;
-2\pm 2 i \sqrt{11}\,.
\]


\vspace{0.2cm}
\noindent
\textbf{Step~3: Numerical study of the solution.}
We numerically integrate the eighth--order differential equation given in Lemma \ref{edo03bis}
with initial conditions prescribed at \(x=\tfrac12\)\,, namely \(y(\tfrac12)=1\) together with the
even derivatives \(y^{(2k)}(\tfrac12)\) computed in Step 1,
the odd derivatives being zero by symmetry. The resulting solution \(y\),
which is even with respect to \(x=\tfrac12\)\,, blows up at both endpoints \(x=0\) and \(x=1\)\,.

\begin{figure}[H]
	\centering
	\includegraphics[width=0.50\textwidth]{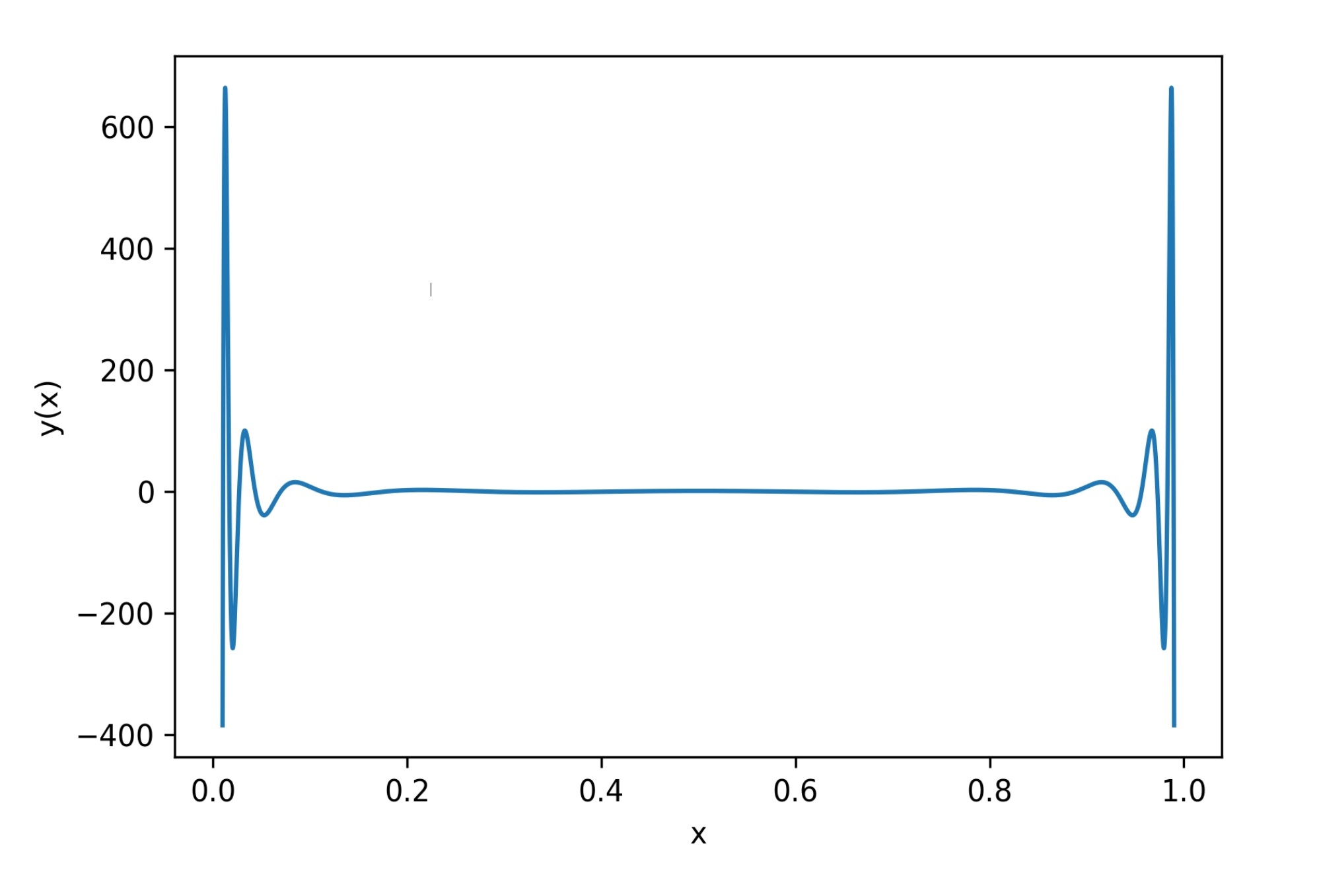}
	\caption{Numerical solution \(y(x)\) on \([10^{-2},\,1-10^{-2}]\)\,.
	The oscillations visible on a logarithmic scale are compatible with the complex
	Frobenius exponents \(\alpha=-2\pm 2i\sqrt{11}\) at the endpoints.}
	\label{fig:y-oscillations-interior}
\end{figure}

\noindent
Guided by these indicial roots, we extract the boundary asymptotics at \(x=0\)
by fitting \(x^2 y(x)\) on shrinking windows near the boundary. We obtain numerically
\[
y(x)
=
\frac{C_0}{x^2}\,\cos\!\bigl(2\sqrt{11}\,\ln x+\phi\bigr)
\;+\;o(x^{-2})\,,
\qquad x\to0^+\,,
\]
with fitted constants
\[
C_0 \approx 0.11325\,,
\qquad
\phi \approx -2.79\,.
\]

\medskip
\noindent
Refining the fit, we construct an explicit boundary approximation
\(y_{\mathrm{app}}\) such that
\[
y - y_{\mathrm{app}} = o(1)\,,
\qquad x \to 0^+ .
\]
Based on the Frobenius analysis below, we introduce the following approximation near \(x=0\), retaining all terms up to order \(o(1)\):
\begin{equation}\label{eq:y-app}
	\begin{aligned}
		y_{\mathrm{app}}(x)
		=\frac{1}{x^{2}}\Bigl[
		&\bigl(A_0 + A_1 x + A_2 x^2 \bigr)
		\cos\!\bigl(2\sqrt{11}\,\ln x\bigr)
		\\
		&+
		\bigl(B_0 + B_1 x + B_2 x^2 \bigr)
		\sin\!\bigl(2\sqrt{11}\,\ln x\bigr)
		\Bigr]\,,
		\qquad x\to0^+ \,.
	\end{aligned}
\end{equation}
The fitted numerical coefficients are
\[
\begin{aligned}
	A_0 &\approx -0.10617307\,,
	&\qquad
	B_0 &\approx \phantom{-}0.03940251\,,
	\\
	|A_1| &\lesssim 10^{-6}\,,
	&\qquad
	|B_1| &\lesssim 10^{-6}\,,
	\\
	A_2 &\approx -0.0923567\,,
	&\qquad
	B_2 &\approx \phantom{-}0.3818868\,.
\end{aligned}
\]


\medskip
\noindent
\textbf{Theoretical Frobenius asymptotics.}
We recall that the indicial equation at \(x=0\) associated with the differential equation in Lemma~\ref{edo03bis} has the following roots:
\[
\rho=-2,\;1,\;3,\;5,\;6,\;7\,,
\qquad
\rho=-2\pm 2i\sqrt{11}\,.
\]
Using \textsc{Mathematica}, we compute the Frobenius expansions associated with
each indicial root.
Truncating the resulting expressions up to terms \(o(1)\) as \(x\to0^+\)\,,
we obtain the following local expansion:
\begin{equation}\label{eq:y-asymptotic-0-clean}
\begin{aligned}
y(x)=\;&
A \Bigl(\frac{120}{x^{2}}-120\Bigr)
\\
&+
B\,x^{-2-2 i\sqrt{11}}
\Bigl(
9375 i+1215\sqrt{11}
+\bigl(6542 i+10474\sqrt{11}\bigr)x^2
\Bigr)
\\
&+
C\,x^{-2+2 i\sqrt{11}}
\Bigl(
-9375 i+1215\sqrt{11}
+\bigl(-6542 i+10474\sqrt{11}\bigr)x^2
\Bigr)
\\
&\;+\;o(1)\,,
\qquad x\to0^+ \,.
\end{aligned}
\end{equation}
for some complex constants \(A,B,C\)\,. Note that the remaining Frobenius solutions associated with the real roots
\(\rho=1,3,5,6,7\) behave like \(O(x)\) or higher powers of \(x\)\,,
and are therefore absorbed into the \(o(1)\) remainder.

\vspace{0.2cm}\noindent
Thus, taking either the real or the imaginary part of the previous 
complex solution $y(x)$ yields a real solution admitting the asymptotic expansion
\begin{equation}\label{eq:y-asymptotic-0}
\begin{aligned}
y(x)
=\;&
C_{-2}\Bigl(\frac{120}{x^{2}}-120\Bigr)
\\
&+\frac{1}{x^{2}}
\Bigl[
\bigl(C_c+\widetilde C_c\,x^{2}\bigr)
\cos\!\bigl(2\sqrt{11}\,\ln x\bigr)
+
\bigl(C_s+\widetilde C_s\,x^{2}\bigr)
\sin\!\bigl(2\sqrt{11}\,\ln x\bigr)
\Bigr]
\\
&\;+\;o(1)\,,
\qquad x\to0^+ \,,
\end{aligned}
\end{equation}
for some real constants \(C_{-2},C_c,C_s,\widetilde C_c,\widetilde C_s\). These constants are not independent: $(C_c,C_s)$ uniquely determines
$(\widetilde C_c,\widetilde C_s)$ through the linear relation
\[
\begin{pmatrix}
	\widetilde C_c\\[1mm]\widetilde C_s
\end{pmatrix}
=
\begin{pmatrix}
	\frac{29}{15} & \frac{13\sqrt{11}}{15}\\[2mm]
	-\frac{13\sqrt{11}}{15} & \frac{29}{15}
\end{pmatrix}
\begin{pmatrix}
	C_c\\[1mm]C_s
\end{pmatrix}.
\]
We also observe that the oscillatory modes do not contain any term of order $x^{-1}\,$.

\medskip
\noindent
\textbf{Comparison with numerics.}
Comparing with the numerical ansatz~\eqref{eq:y-app}, we obtain the identifications
\[
A_0=C_c\,,\qquad B_0=C_s\,,\qquad A_2=\widetilde C_c\,,\qquad B_2=\widetilde C_s\,.
\]
Moreover, the Frobenius structure predicts that
\[
A_1=B_1=0\,.
\]
The fitted numerical values satisfy \(|A_1|,|B_1|\lesssim 10^{-6}\),
which is fully compatible with this theoretical prediction.
Any nonzero values obtained at lower precision should therefore be interpreted
as numerical artefacts.

\smallskip
\noindent
The numerical fit shows no evidence of a non-oscillatory \(x^{-2}\) contribution,
which suggests
\[
C_{-2}=0\,.
\]
Identifying the leading oscillatory coefficients yields
\[
C_c \approx A_0 \approx -0.106173\,,
\qquad
C_s \approx B_0 \approx \phantom{-}0.039403\,.
\]
Equivalently, writing
\[
C_c\,\cos\!\bigl(2\sqrt{11}\,\ln x\bigr)
+
C_s\,\sin\!\bigl(2\sqrt{11}\,\ln x\bigr)
=
C_0\,\cos\!\bigl(2\sqrt{11}\,\ln x+\phi\bigr)\,,
\]
we recover
\[
C_0=\sqrt{C_c^2+C_s^2}\approx 0.11325\,,
\qquad
\phi=\arctan\!\Bigl(-\frac{C_s}{C_c}\Bigr)\approx -2.79\,.
\]

\vspace{0.2cm}\noindent
Finally, the Frobenius expansion implies that the constant-order oscillatory coefficients
\(\widetilde C_c,\widetilde C_s\) are not independent. Using the fitted values of \(C_c\) and \(C_s\)\,, this relation predicts
\[
\widetilde C_c^{\mathrm{pred}} \approx -0.0917\,,
\qquad
\widetilde C_s^{\mathrm{pred}} \approx \phantom{-}0.3819\,.
\]
These values are in excellent agreement with the numerical coefficients
\[
\widetilde C_c \approx A_2 \approx -0.09236\,,
\qquad
\widetilde C_s \approx B_2 \approx \phantom{-}0.38189\,,
\]
with relative discrepancies of approximately \(0.7\, \%\) for \(\widetilde C_c\)
and \(3\times10^{-4} \, \% \) for \(\widetilde C_s\)\,.
This quantitative agreement provides strong numerical evidence
for the presence of the quadratic correction in the oscillatory
Frobenius modes.

\medskip
\noindent
\textbf{Role of the numerical analysis.}
At this stage, it is important to clarify the role played by the numerical computations.
They are used solely to detect that the solution \(y\) of the eighth--order equation
is not bounded near \(x=0\) and \(x=1\)\,.
This observation implies that at least one of the singular Frobenius modes
associated with the indicial roots
\[
\rho=-2
\qquad\text{or}\qquad
\rho=-2\pm 2 i\sqrt{11}
\]
is present in the expansion of \(y\) near the endpoints.
No numerical information on the corresponding coefficients is required for this conclusion.

\smallskip
\noindent
In particular, the proof that the primitive \(\zeta\) does not belong to \(L^2(0,1)\)
relies only on the qualitative oscillatory behaviour induced by these Frobenius modes,
and not on the explicit values of the constants appearing in the asymptotic expansion.
The numerical computation of the coefficients
\((C_c,C_s,\widetilde C_c,\widetilde C_s)\)
should therefore be viewed as an additional quantitative illustration,
confirming the Frobenius structure predicted by the theory.

\medskip
\noindent
\textbf{Non--square integrability of the primitive.}
Let \(\zeta\) be a primitive of \(y\), namely
\[
\zeta(x):=\int_{1/2}^{x} y(s)\,ds\, ,
\qquad x\in(0,1).
\]
At this stage, and for the sake of generality, we do not assume that \(C_{-2}=0\)\,.
Integrating the leading terms in the asymptotic expansion of \(y\) as \(x\to0^+\)
yields
\[
\zeta(x)
=
\frac{-120\,C_{-2}}{x}
+
\frac{C_0}{45}\,
\frac{
	-\cos\!\bigl(2\sqrt{11}\,\ln x + \phi\bigr)
	+ 2\sqrt{11}\,\sin\!\bigl(2\sqrt{11}\,\ln x + \phi\bigr)
}{x}
+ O(1)\,,
\qquad x\to0^+\,.
\]
Setting \(t=-\ln x\)\,, we can rewrite this behaviour as
\[
\zeta(e^{-t}) = e^{t}\,A(t) + O(1)\,,
\qquad t\to+\infty,
\]
where the function \(A\) is bounded, periodic, and explicitly given by
\[
A(t)
:=
-120\,C_{-2}
+
\frac{C_0}{45}
\Bigl(
-\cos\!\bigl(2\sqrt{11}\,t-\phi\bigr)
-2\sqrt{11}\,\sin\!\bigl(2\sqrt{11}\,t-\phi\bigr)
\Bigr)\,.
\]
Since \(A\not\equiv 0\) and \(A\) is continuous and periodic, there exist constants
\(\delta>0\) and \(c_0>0\)\,, and a sequence \(t_n\to+\infty\)\,, such that
\[
|A(t)|\ge c_0
\qquad\text{for all }t\in I_n:=[t_n,t_n+\delta]\,.
\]
Consequently,
\[
\int_0^1 |\zeta(x)|^2\,dx
=
\int_0^{+\infty} |\zeta(e^{-t})|^2\,e^{-t}\,dt
=
\int_0^{+\infty} e^{t}\,|A(t)|^2\,dt
=+\infty\,.
\]
because
\[
\int_0^{+\infty} e^{t}\,|A(t)|^2\,dt
\;\ge\;
\sum_{n}\int_{I_n} e^{t}\,|A(t)|^2\,dt
\;\ge\;
c_0^2\sum_{n}\int_{t_n}^{t_n+\delta} e^{t}\,dt
=+\infty\,.
\]
This proves that \(\zeta\notin L^2(0,1)\)\,.

\vspace{0.2cm}\noindent
\medskip
\noindent
\textbf{Step~4: Conclusion.}
We therefore conclude that one must take \( y(\tfrac12)=0 \)\,.
By the triangular system \((S)\), all derivatives of \( y \) at \( x=\tfrac12 \) up to order \(7\)
vanish. The Cauchy--Lipschitz theorem then implies that \( y \equiv 0 \) on \( (0,1) \).
Consequently, \( \zeta \equiv 0 \) on \( (0,1) \)\,.

\vspace{0.3cm}
\noindent
In particular, this implies that the differential of the spectral map associated with the pair \( (\ell_1,\ell_2)=(0,3) \) is injective\,.



\section{Conclusion}

The techniques developed in this paper rely on the Kneser--Sommerfeld formula.
In the case $(\ell_1,\ell_2)=(0,3)\,$, the analysis required the use of computer-assisted
numerical computations.
Even with the assistance of computer algebra, the study of the resulting differential
equations already becomes rapidly cumbersome in the case $(\ell_1,\ell_2)=(1,3)\,$. 
Larger values of $(\ell_1,\ell_2)$ currently seem out of reach, which suggests that
a more conceptual approach is needed.

\vspace{0.2cm}\noindent
We are nevertheless led to the following natural conjecture, originally formulated by Rundell and Sacks: \emph{local uniqueness holds near the zero potential as soon as the Dirichlet spectra are known for two distinct values of \(\ell\)}.


\appendix
\section{Proof of Theorem~\ref{thm:global}}
\label{app:A} 
\renewcommand{\thethm}{A.\arabic{thm}}

\noindent
In this appendix, we briefly recall the scattering-theoretic framework underlying 
Theorem~\ref{thm:global}. 
Throughout this section, we fix the energy $\lambda \in \mathbb{C}$ and consider the radial Schr\"odinger equation
\begin{equation}\label{eq:radial_scattering}
- u''(r) + \left( \frac{\ell(\ell+1)}{r^2} + \hat{q}(r) \right) u(r) = \lambda\, u(r)\,,
\qquad r >0,
\end{equation}
where $\hat{q}$ denotes the extension of $q \in L^2 (0,1)$ by zero for $r \ge 1$.\footnote{In the remainder of this appendix, we omit the hat notation and use the same notation $q$ for the potential in $L^2(0,1)$  and its extension by $0$ in $L^2(0,+\infty)$\,.}

\vspace{0.2cm}\noindent
For each \(\nu = \ell + \tfrac12\), we denote by 
\(\varphi(r,\nu,\lambda)\) the \emph{regular solution} of \eqref{eq:radial_scattering}, 
characterized by the boundary condition
\begin{equation}
\varphi(r,\nu,\lambda) \sim r^{\nu+\frac12}\,, \qquad r \to 0\,.
\end{equation}
 By definition, a solution of~(B.1) is a \(C^1\)-function whose derivative with respect to \(r\) is locally absolutely continuous on \((0,+\infty)\).
\vspace{0.2cm}\noindent
For instance, when \(q \equiv 0\), the regular solution is explicitly given by
\begin{equation}
\varphi_0(r,\nu,\lambda)
= 2^{\nu}\, \lambda^{-\nu/2}\, \Gamma(\nu+1)\, \sqrt{r}\, J_{\nu}(\sqrt{\lambda}\, r)\,,
\end{equation}
where \(\sqrt{\lambda}\) is defined by the principal branch.\footnote{By convention, if 
\(\lambda = |\lambda| e^{i\theta}\) with \(\theta \in (-\pi,\pi]\)\,, 
then \(\sqrt{\lambda} = \sqrt{|\lambda|}\, e^{i\theta/2}\)\,.}

\vspace{0.2cm}\noindent
The existence and uniqueness of a regular solution are ensured by the following theorem, 
proved in~\cite[Lemma~2.2]{KST10} (see also~\cite[Lemma~B.2]{KoTe13}) for locally integrable potentials \(q\) on \((0,+\infty)\) 
under the additional assumption \(\int_0^1 r\,|q(r)|\,dr < \infty\,.\)

\begin{thm}\label{thm:solreg}
There exists a unique regular solution \(\varphi(r,\nu,\lambda)\) of 
\eqref{eq:radial_scattering} satisfying:
\begin{itemize}
    \item[(i)] The map \(\lambda \mapsto \varphi(r,\nu,\lambda)\) is an entire function of order $\frac12$\,.
    \item[(ii)] There exists $C>0$ such that for all  \(r>0\) and $\lambda \in \C$\,,
    \begin{equation}
    |\varphi(r,\nu,\lambda)| \;\le\; 
    C \left( \frac{r}{1 + |\sqrt{\lambda}|\,r} \right)^{\ell+1} 
    e^{|\Im \sqrt{\lambda}|\,r}\,.
    \end{equation}
    \item[(iii)] Moreover, there exists $C>0$ such that for all  \(r >0\) and  $\lambda \in \C$\,,
    \begin{equation}
    \bigl|\varphi(r,\nu,\lambda) - \varphi_0(r,\nu,\lambda)\bigr|
    \;\le\;
    C \left( \frac{r}{1 + |\sqrt{\lambda}|\,r} \right)^{\ell+1} 
    e^{|\Im \sqrt{\lambda}|\,r}
    \int_0^r \frac{t\,|q(t)|}{1 + |\sqrt{\lambda}|\,t}\,dt\,.
    \end{equation}
\end{itemize}
\end{thm}

\vspace{0.2cm}\noindent
The \emph{Dirichlet spectrum} associated with a fixed angular momentum~$\ell$ 
(or equivalently $\nu = \ell + \tfrac12$) is defined as the set of values 
\(\{\lambda_{\ell,n}(q)\}_{n\ge1}\) such that the regular solution satisfies 
the boundary condition 
\begin{equation}
\varphi(1,\nu,\lambda_{\ell,n}(q)) = 0\,.
\end{equation}
Equivalently, the Dirichlet spectrum consists of the zeros of the entire function 
\(\lambda \mapsto \varphi(1,\nu,\lambda)\)\,. 

\vspace{0.2cm}\noindent
It is well known that, for such potentials $q$ and for each fixed angular momentum $\ell \in \mathbb{N}$, the associated Dirichlet spectrum forms a countable sequence of simple, real eigenvalues $\{\lambda_{\ell,n}(q)\}_{n\ge 1}$ diverging to infinity. These facts are classical and can be found in~\cite{CarlShu94,Carlson97}. 

\vspace{0.2cm}\noindent
For instance, when \(q \equiv 0\), the Dirichlet eigenvalues are explicitly given by 
\begin{equation}
\lambda_{\ell,n}(0) = j_{\nu,n}^2\,,
\end{equation}
where \(j_{\nu,n}\) denotes the \(n\)-th positive zero of \(J_\nu(z)\). Recall that the Bessel function \(J_\nu(z)\) admits the classical Hadamard factorization
\begin{equation}
J_\nu(z) 
= \frac{(z/2)^{\nu}}{\Gamma(\nu+1)}
\prod_{n=1}^{\infty} \left(1 - \frac{z^2}{j_{\nu,n}^2}\right)\,.
\end{equation}
As a consequence, the free regular solution evaluated at \(r=1\) can be written as
\begin{equation}
\varphi_0(1,\nu,\lambda)
= \prod_{n=1}^{\infty} \left(1 - \frac{\lambda}{j_{\nu,n}^2}\right)\,.
\end{equation}

\vspace{0.2cm}\noindent
Similarly, by applying Hadamard's factorization theorem, 
\(\varphi(1,\nu,\lambda)\) admits the following canonical product representation:

\begin{lemma}
\label{lem:factorization}
For each fixed \(\nu = \ell + \tfrac12\) and \(q \in L^2(0,1)\)\,, the regular solution satisfies
\[
\varphi(1,\nu,\lambda)
= C_\nu(q) \ \lambda^m \prod_{n=1}^{\infty}
\left(1 - \frac{\lambda}{\lambda_{\ell,n}(q)}\right)\,,
\]
where \(C_\nu(q) \neq 0\) is a constant depending on \(\nu\) and \(q\), and \(m \in \{0,1\}\) denotes the order of the zero at \(\lambda = 0\).
\end{lemma}

\vspace{0.2cm}\noindent
We next show that equality of the Dirichlet spectra implies equality of the regular solutions evaluated at \(r=1\)\,.

\begin{lemma}
\label{lem:factorization_equality}
Let $\nu = \ell + \tfrac12$ and let  $q\,,\, \tilde q \in L^2 (0,1)$ be two potentials such that
their Dirichlet spectra coincide:
\begin{equation}
\{\lambda_{\ell,n}(q)\}_{n\ge1} = \{\lambda_{\ell,n}(\tilde{q})\}_{n\ge1}\,.
\end{equation}
Then the corresponding regular solutions satisfy
\begin{equation}
\varphi(1,\nu,\lambda;q) = \varphi(1,\nu,\lambda;\tilde q)\,, \qquad \forall\, \lambda \in \mathbb{C}\,.
\end{equation}
\end{lemma}

\begin{proof}
We denote by a tilde all quantities associated with the potential~$\tilde q$, and, 
for brevity, we write $\varphi(1,\nu,\lambda;q)$ simply as $\varphi(1,\nu,\lambda)$. Since the Dirichlet spectra coincide, we have $m = \tilde m$\,. By Theorem~\ref{thm:solreg} (with \(r=1\)), for real \(\lambda = k^2\) and \(|k|\to\infty\)\,,
\begin{equation}
\bigl|\varphi(1,\nu,\lambda) - \varphi_0(1,\nu,\lambda)\bigr|
= \mathcal{O}(|k|^{-\ell-2})\,.
\end{equation}
From the standard asymptotics of Bessel functions, one has as $\lambda \to + \infty$,
\begin{equation}
\varphi_0(1,\nu,\lambda)
\sim c_\nu\, k^{-\ell-1}
\cos\!\Big(k - \tfrac{\pi\nu}{2} - \tfrac{\pi}{4}\Big)\,,
\qquad k = \sqrt{\lambda} \in \mathbb{R}\,.
\end{equation}
Choosing 
\begin{equation}
k_n = \tfrac{\pi}{2}\Big(\nu + \tfrac12\Big) + 2\pi n\,, 
\qquad \lambda_n = k_n^2\,,
\end{equation}
with $n \in \N$\,, so that \(\cos(\cdot)=1\)\,, we obtain as $n \to + \infty$\,,
\begin{equation}
\varphi(1,\nu,\lambda_n)
\sim \varphi_0(1,\nu,\lambda_n)
\sim \tilde{\varphi}(1,\nu,\lambda_n)\,.
\end{equation}
Therefore, using Lemma \ref {lem:factorization}, the leading coefficients coincide, that is, \(C_\nu(q) = C_\nu(\tilde q)\), which completes the proof.
\end{proof}

\vspace{0.5cm}\noindent
One can also define the 
\emph{Jost solutions} \(f^{\pm}(r,\nu,\lambda)\) as the unique \(C^1\) solutions with respect to \(r\) of 
\eqref{eq:radial_scattering} satisfying the oscillatory asymptotics
\begin{equation}
f^{\pm}(r,\nu,\lambda) \sim e^{\pm i \sqrt{\lambda}\, r}\,, 
\qquad r \to +\infty\,.
\end{equation}
The proof of this result is given in~\cite[Lemma~B.4]{KoTe13} under the assumptions 
\(\int_0^1 r\,|q(r)|\,dr < \infty\) and \(\int_1^{+\infty} |q(r)|\,dr < \infty\,.\)
In particular, for \(\lambda = 1\) (which is the case considered in~\cite{DaNi16} and~\cite{Ra99}), and for such potentials supported in \([0,1]\), one has, for all \(r \ge 1\) 
(see~\cite[Lemma~B.3]{KoTe13}),
\begin{equation}\label{fp0}
f^{+}(r,\nu,1)
= e^{i(\nu+\frac12)\frac{\pi}{2}}\,
\sqrt{\frac{\pi r}{2}}\,
H_{\nu}^{(1)}(r)
\end{equation}
and
\begin{equation}\label{fm0}
f^{-}(r,\nu,1)
= e^{-i(\nu+\frac12)\frac{\pi}{2}}\,
\sqrt{\frac{\pi r}{2}}\,
H_{\nu}^{(2)}(r)\,,
\end{equation}
where $H_{\nu}^{(j)}(r)$ denotes the Hankel function of order~$\nu$.  It is a fundamental observation that the Jost solutions are \emph{entirely independent of the potential} for all \(r \ge 1\)\,.

\vspace{0.2cm}\noindent
The pair of Jost solutions \( \{f^{+}(r,\nu,1),\, f^{-}(r,\nu,1)\} \) forms a fundamental system of solutions of~\eqref{eq:radial_scattering}. 
Hence, at the fixed energy \(\lambda = 1\), the regular solution can be written as
\begin{equation}\label{SFS}
\varphi(r,\nu,1) = \alpha(\nu)\, f^{+}(r,\nu,1) + \beta(\nu)\, f^{-}(r,\nu,1)\,,
\end{equation}
where \(\alpha(\nu),\, \beta(\nu) \in \mathbb{C}\) are called the \emph{Jost functions}.

\vspace{0.2cm}\noindent
We now recall some basic properties of the Jost functions 
(see~\cite{DaNi16} for details). 
They satisfy the conjugation relation 
\(\overline{\alpha(\nu)} = \beta(\nu)\)\,, 
and neither \(\alpha(\nu)\) nor \(\beta(\nu)\) vanishes. 
This allows us to introduce the \emph{Regge interpolation}:
\begin{equation}\label{interpolation}
\sigma(\nu) = e^{i\pi(\nu + \frac12)}\, \frac{\alpha(\nu)}{\beta(\nu)}\,.
\end{equation}
Note that, following~\cite{Re59}, the so-called \emph{phase shifts}~$\delta(\nu)$ are defined as a continuous function of~$\nu \in (0,+\infty)$ through the relation
\[
\sigma(\nu) = e^{2i\,\delta(\nu)}\,.
\]
The phase shifts become uniquely determined by imposing the condition $\delta(\nu) \to 0$ as $\nu \to +\infty$, and they can be  meromorphically  continued to complex values of~$\nu$\,.

\vspace{0.2cm}\noindent
Consider now a sequence \((\ell_k)_{k\ge1}\) of positive integers satisfying 
\begin{equation}
\sum_{k} \frac{1}{\ell_k} = +\infty\,,
\end{equation}
that is, the \emph{M\"untz condition}. 
Then the following uniqueness result holds (see Theorem~2.1 in~\cite{Ra99}):  
if \(q, \tilde q \in L^2(0,1)\) are two potentials whose Regge interpolation functions 
coincide for \(\nu_k = \ell_k + \tfrac12\)\,, then 
\begin{equation}
q = \tilde q \quad \text{a.e. on } (0,1)\,.
\end{equation}

\vspace{0.2cm}\noindent
Assume now that \(q\) and \(\tilde q\) share the same Dirichlet spectra 
\(\{\lambda_{\ell_k,n}(q)\}_{k,n \ge 1}\)\,. 
Since the Jost solutions \(f^{\pm}(r,\nu,1)\) never vanish 
(see~\cite[Lemma~4.2]{DaNi16}) and are \emph{independent of the potential at $r=1$}\,, 
it follows from~\eqref{SFS} and Lemma~\ref{lem:factorization_equality} 
that the corresponding Regge interpolation functions 
coincide for \(\nu_k = \ell_k + \tfrac{1}{2}\)\,. 
This concludes the proof of Theorem~\ref{thm:global}.



\section{Proof of Theorem~\ref{Closed range} }\label{app:B}
\renewcommand{\thethm}{B.\arabic{thm}}

In this section, we prove that the range of the differential of the spectral map
at the zero potential is closed in the cases $(\ell_1,\ell_2) = (0,1)\,, \  (1,2)$ and $(0,3)$\,.
Moreover, in the case $(0,1)$, we give an explicit description of this range.

\vspace{0.2cm}\noindent
We first recall that in both cases $(0,1)$\,, $(1,2)$ and $(0,3)$\,, the differential of the
spectral map at $q=0$ is injective.
Equivalently, the intersection of the tangent spaces to the corresponding
isospectral manifolds at $q=0$ is trivial.
Following the terminology of Shubin Christ \cite{Sh} and Carlson--Shubin \cite{CarlShu94}, we introduce
\[
W \;:=\; T_0\mathcal{M}_{\ell_1}\;\cap\;T_0\mathcal{M}_{\ell_2}\,,
\]
where $\mathcal{M}_{\ell}$ denotes the isospectral set associated with the
Dirichlet spectrum for angular momentum $\ell$\,.
Injectivity of the differential is precisely the statement that
\[
W = \{0\}\,.
\]
Thus, in the case $(\ell_1,\ell_2)=(0,1)$\,, the first part of Theorem~\ref{Closed range} follows from (\cite{Sh}, Lemma~5), while for $(\ell_1,\ell_2)=(1,2)$ or $(0,3)$, it follows from the analysis
carried out in the proof of (\cite{CarlShu94}, Theorem~5.6). 

\vspace{0.2cm}\noindent
We have thus shown that in the cases $(\ell_1,\ell_2) = (0,1)\,$, $(\ell_1,\ell_2) = (1,2)\,$ and $(\ell_1,\ell_2) = (0,3)$\,,
the differential of the spectral map at $q=0$ is injective and has closed range.
In other words, the operator $d_0 \mathcal{S}_{\ell_1,\ell_2}$ is semi-Fredholm
(see~\cite[p.~230]{Kato80}).
Moreover, if an operator $B$ is Fredholm or semi-Fredholm, then any compact perturbation
remains Fredholm or semi-Fredholm and the Fredholm index is preserved
(see~\cite[Theorem~5.26]{Kato80}).
\vspace{0.2cm}\noindent
We now use these theoretical results to refine the analysis in the case $(0,1)$. 
To this end, we set
\[
\Phi_\ell(x) := \frac{\pi}{2}\, x\, J_{\ell+\frac12}(x)^2 = (j_\ell (x))^2\, .
\]
First, we decompose the differential $d_0 \mathcal{S}_{\ell_1,\ell_2}$ into two parts,
\[
d_0 \mathcal{S}_{\ell_1,\ell_2} = A_{\ell_1,\ell_2} + K_{\ell_1,\ell_2}\,,
\]
where we define
\begin{align}
	A_{\ell_1,\ell_2}(\zeta)
	= \Bigg(
	\langle \zeta, 1 \rangle,
	\;&\bigl(
	\langle \zeta,\; 2\Phi_{\ell_1}\!\left(\bigl(n+\tfrac{\ell_1}{2}\bigr)\pi\,\cdot\right) - 1 \rangle
	\bigr)_{n\geq 1}\,,
	\nonumber\\[0.2cm]
	&\bigl(
	\langle \zeta,\; 2\Phi_{\ell_2}\!\left(\bigl(n+\tfrac{\ell_2}{2}\bigr)\pi\,\cdot\right) - 1 \rangle
	\bigr)_{n\geq 1}
	\Bigg)\,,
\end{align}
and
\begin{align}
	K_{\ell_1,\ell_2}(\zeta)
	= \Bigg(
	0,
	\;&\bigl(
	\langle \zeta,\; g_{\ell_1,n}^2
	- 2\Phi_{\ell_1}\!\left(\bigl(n+\tfrac{\ell_1}{2}\bigr)\pi\,\cdot\right)
	\rangle
	\bigr)_{n\geq 1}\,,
	\nonumber\\[0.2cm]
	&\bigl(
	\langle \zeta,\; g_{\ell_2,n}^2
	- 2\Phi_{\ell_2}\!\left(\bigl(n+\tfrac{\ell_2}{2}\bigr)\pi\,\cdot\right)
	\rangle
	\bigr)_{n\geq 1}
	\Bigg)\,.
\end{align}
According to Theorem~2.1 and Corollary~2.2 in Serier~\cite{Ser07}, one has
\[
\| g_{\ell,n}^2 - 2\Phi_\ell(j_{\nu,n}\,\cdot) \|_{L^2(0,1)} =\mathcal O\!\left(\frac{1}{n}\right)\,.
\]
Since $|\Phi'_{\ell}(x)| \leq C_\ell$ uniformly on $\mathbb{R}$ and
\[
\sum_{n\geq 1} \bigl| j_{\ell+\frac12,n} - (n+\tfrac{\ell}{2})\pi \bigr|^2 < \infty
\qquad \text{for any integer } \ell\,,
\]
it follows that
\[
\sum_{n\geq 1}
\bigl\| g_{\ell,n}^2
- 2\Phi_{\ell}\bigl((n+\tfrac{\ell}{2})\pi\,\cdot\bigr)
\bigr\|_{L^2(0,1)}^2
< \infty \,.
\]
Therefore, by the Cauchy--Schwarz inequality, the operator
$K_{\ell_1,\ell_2}$ is Hilbert-Schmidt, and hence compact, from
$L^2(0,1)$ into $\mathbb{R} \times \ell^2(\mathbb{N}) \times \ell^2(\mathbb{N})$\,.
As a consequence of the stability of semi-Fredholm operators under compact
perturbations, it follows that, for $(\ell_1, \ell_2)= (0,1)$, $(\ell_1,\ell_2) = (1,2)\,$ or $(0,3)$ the range of
$A_{\ell_1,\ell_2}\,$, which is a compact perturbation of
$d_0 \mathcal{S}_{\ell_1,\ell_2}\,$, is also closed.
\vspace{0.2cm}\noindent
Now, according to~\cite[Corollary~5.2]{RuSa01}, in the case $(\ell_1,\ell_2) = (0,1)\,$, the family
\[
\left(
1,\;
\bigl(\Phi_{\ell_1}((n+\tfrac{\ell_1}{2})\pi\,\cdot)\bigr)_{n\geq1},\;
\bigl(\Phi_{\ell_2}((n+\tfrac{\ell_2}{2})\pi\,\cdot)\bigr)_{n\geq1}
\right)
\]
is referred to as a \emph{basis} of $L^2(0,1)$\footnote{This statement should be interpreted with care.
In the sense of Gohberg-Krein~\cite{GohbergKrein69}, this only means that the above family is complete and $\omega$-linearly independent in $L^2(0,1)$
This property does not, in general, imply that this family is a Krein basis or a Riesz basis.}. We define
\[
\mathcal F_{\ell_1,\ell_2}
=
\Bigl(
f_0,\ (f_{1,n})_{n\ge1}\,,\ (f_{2,n})_{n\ge1}
\Bigr)\subset L^2(0,1)\,,
\]
where
\[
f_0(r)=1,\qquad
f_{1,n}(r)=2\Phi_{\ell_1}\!\left(\bigl(n+\tfrac{\ell_1}{2}\bigr)\pi r\right)-1\,,\qquad
f_{2,n}(r)=2\Phi_{\ell_2}\!\left(\bigl(n+\tfrac{\ell_2}{2}\bigr)\pi r\right)-1\,.
\]
The family $\mathcal F_{\ell_1,\ell_2}$ remains complete and algebraically linearly independent
in $L^2(0,1)$\,. With this notation, the operator $A_{\ell_1,\ell_2}$  can be written as
\[
A_{\ell_1,\ell_2}(\zeta)
=
\Bigl(
\langle \zeta,f_0\rangle,\ (\langle \zeta,f_{1,n}\rangle)_{n\ge1},\
(\langle \zeta,f_{2,n}\rangle)_{n\ge1}
\Bigr)\,,
\]
and is therefore injective. Recall that $c_{00}$ denotes the space of finitely supported sequences.
Since
\[
\mathbb{R}\times c_{00}\times c_{00}
\quad\text{is dense in}\quad
\mathbb{R} \times \ell^2_{\mathbb{R}}(\mathbb{N}) \times \ell^2_{\mathbb{R}}(\mathbb{N})\,,
\]
it follows from a straightforward algebraic argument that the range
 of $A_{\ell_1,\ell_2}$ is also dense in $\mathbb{R} \times \ell^2_{\mathbb{R}}(\mathbb{N}) \times \ell^2_{\mathbb{R}}(\mathbb{N})$.
Since this range is closed, as shown above, we conclude that
\[
\operatorname{Ran} A_{\ell_1,\ell_2}
=
\mathbb{R} \times \ell^2_{\mathbb{R}}(\mathbb{N}) \times \ell^2_{\mathbb{R}}(\mathbb{N})\,.
\]
In particular, the operator $A_{\ell_1,\ell_2}$ is Fredholm of index zero.

\vspace{0.2cm}\noindent
Finally, since $d_0\mathcal{S}_{0,1}$ is a compact perturbation of $A_{0,1}$, it follows that
$d_0\mathcal{S}_{0,1}$ is also Fredholm with the same index.
Because it is injective, it must therefore be an isomorphism from $L^2(0,1)$ onto
$\mathbb{R} \times \ell^2_{\mathbb{R}}(\mathbb{N}) \times \ell^2_{\mathbb{R}}(\mathbb{N})$\,.
This completes the proof of Theorem~\ref{Closed range}.

\vspace{0.5cm}
	
	\noindent \footnotesize{
		
		\noindent Laboratoire de Math\'ematiques Jean Leray, UMR CNRS 6629. Nantes Universit\'e  F-44000 Nantes  \\
		\emph{Email adress}: damien.gobin@univ-nantes.fr \\
		
		\noindent Laboratoire de Math\'ematiques Jean Leray, UMR CNRS 6629. Nantes Universit\'e  F-44000 Nantes  \\
		\emph{Email adress}: benoit.grebertr@univ-nantes.fr \\
		
		\noindent Laboratoire de Math\'ematiques Jean Leray, UMR CNRS 6629. Nantes Universit\'e  F-44000 Nantes  \\
		\emph{Email adress}: Bernard.Helffer@univ-nantes.fr \\
		
		\noindent Laboratoire de Math\'ematiques Jean Leray, UMR CNRS 6629. Nantes Universit\'e  F-44000 Nantes \\
		\emph{Email adress}: francois.nicoleau@univ-nantes.fr \\


\begin{thebibliography}{9}


\bibitem{Abraham88}
R.~Abraham, J.~E.~Marsden, and T.~Ratiu,
\textit{Manifolds, Tensor Analysis, and Applications},
Applied Mathematical Sciences, Vol.~75,
Springer, 1988.


\bibitem{AS64}
	M.~Abramowitz and I.~A.~Stegun, \emph{Handbook of Mathematical Functions with Formulas, Graphs, and Mathematical Tables}, National Bureau of Standards, Applied Mathematics Series \textbf{55} (1964).
	
	
\bibitem{Buch47a} H.~Buchholz, \emph{Bemerkungen zu einer Entwicklungsformel aus der Theorie der Zylinderfunktionen}, Z. Angew. Math. Mech. \textbf{25/27} (1947), 245--252.


\bibitem{Carslaw14} H.~S.~Carslaw, \emph{The Green function for the equation \( \nabla^2 u + k^2 u = 0 \)}, Proc. Lond. Math. Soc. (2) \textbf{13} (1914), 236--257.

\bibitem{CarlShu94} R.~Carlson and C.~Shubin, 
\textit{Spectral rigidity for spherically symmetric potentials}, 
J. Differential Equations \textbf{113} (1994), no.~2, 273--289.

\bibitem{Carlson97}
R.~Carlson,
\textit{A Borg--Levinson Theorem for Bessel Operators},
Pacific J. Math. \textbf{177} (1997), no. 1, 1--26.


\bibitem{DaNi16}
T. Daud\'e and F.~Nicoleau,
\textit{Local inverse scattering at a fixed energy for radial Schr\"odinger operators},
Ann. Henri Poincar\'e \textbf{17} (2016), no.~1, 29--54.








\bibitem{DuPeVa24} A.~J.~Dur\'an, M.~P\'erez, and J.~L.~Varona, \emph{Summing Sneddon-Bessel series explicitly}, Math. Meth. Appl. Sci. \textbf{47} (2024), 6590--6606.


\bibitem{Go18} D. Gobin, \emph{Inverse scattering at fixed energy for radial magnetic Schr\"odinger operators}, Ann. Henri Poincar\'e, 19 (2018), 1283--1311.


\bibitem{GohbergKrein69}
I.~C.~Gohberg and M.~G.~Krein,
\textit{Introduction to the Theory of Linear Nonselfadjoint Operators},
Translations of Mathematical Monographs, Vol.~18,
American Mathematical Society, 1969.


\bibitem{GR}
J.~Guillot and J.~Ralston, 
\textit{Inverse spectral theory for a singular Sturm--Liouville operator on [0,1]}, 
J. Differential Equations \textbf{76} (1988), no.~2, 353--373, corrigendum ibid. 250, No. 2, 1232-1233 (2011).

 \bibitem{Kato80}
T.~Kato,
\textit{Perturbation Theory for Linear Operators},
2nd ed.,
Grundlehren der Mathematischen Wissenschaften, Vol.~132,
Springer, 1980.


\bibitem{Ko81} M.~Kobayashi, \emph{Correction of the Kneser--Sommerfeld Expansion Formula}, J. Phys. Soc. Jpn. \textbf{50} (1981), 1391--1395.

\bibitem{Kneser07} A.~Kneser, \emph{Die Theorie der Integralgleichungen und die Darstellung willk\"urlicher Funktionen in der mathematischen Physik}, Math. Ann. \textbf{63} (1907), 477--524.

\bibitem{KST10}
A.~Kostenko, A.~Sakhnovich, and G.~Teschl,
\emph{Inverse eigenvalue problems for perturbed spherical Schr\"odinger operators},
Inverse Problems \textbf{26} (2010), 105013.

\bibitem{KoTe13} 
A.~Kostenko and G.~Teschl, 
\emph{Spectral asymptotics for perturbed spherical Schr\"odinger operators and applications to quantum scattering}, Comm. Math. Phys., 322 (2013), 255--275.


\bibitem{Lebedev72} N.~N.~Lebedev, \textit{Special Functions and Their Applications}, Dover Publications, 1972. (Originally published in Russian in 1963)

\bibitem{MOS1966}
W.~Magnus, F.~Oberhettinger and R.~P.~Soni,
\textit{Formulas and Theorems for the Special Functions of Mathematical Physics},
3rd enlarged ed.,
Grundlehren der Mathematischen Wissenschaften, Vol.~52,
Springer, 1966.


\bibitem{Martin21} P.~A.~Martin, \emph{On Fourier--Bessel series and the Kneser--Sommerfeld expansion}, Math. Meth. Appl. Sci. \textbf{45} (2022), 1145--1152.

\bibitem{Muntz12} Ch.~H.~M\"untz, \emph{\"Uber den Approximationssatz von Weierstrass}, Verhandlungen des Internationalen Mathematiker-Kongresses, ICM Stockholm \textbf{1} (1912), 256--266.

\bibitem{Muntz14} Ch.~H.~M\"untz, \emph{\"Uber den Approximationssatz von Weierstrass}, in: C.~Carath\'eodory, G.~Hessenberg, E.~Landau, L.~Lichtenstein (eds.), \emph{Mathematische Abhandlungen Hermann Amandus Schwarz zu seinem f\"unfzigj\"ahrigen Doktorjubil\"aum}, Springer, Berlin (1914), 303--312.


\bibitem{DLMF}
{\it NIST Digital Library of Mathematical Functions}.
\newblock https://dlmf.nist.gov/, Release 1.2.0 of 2024-03-15.
\newblock F.~W.~J. Olver, A.~B. {Olde Daalhuis}, D.~W. Lozier, B.~I. Schneider,
R.~F. Boisvert, C.~W. Clark, B.~R. Miller, B.~V. Saunders, H.~S. Cohl, and
M.~A. McClain, eds. 





\bibitem{PT}
J.~P\"oschel and E.~Trubowitz, 
\textit{Inverse Spectral Theory}, 
Pure and Applied Mathematics, vol. 130, Academic Press, 1987.


\bibitem{Ra99}
A.~G.~Ramm, 
\emph{An inverse scattering problem with part of the fixed-energy phase shifts}, 
Commun. Math. Phys. \textbf{207} (1999), 231--247.

\bibitem{Re59}
T.~Regge, 
\emph{Introduction to complex orbital momenta}, 
Nuovo Cimento \textbf{14} (1959), no.~5, 951--976.





\bibitem{RuSa92} W.~Rundell and P.~Sacks, \emph{Reconstruction techniques for classical inverse Sturm--Liouville problems}, Math. Comput. \textbf{58} (1992), no. 197, 161--183.

\bibitem{RuSa01} W.~Rundell and P.~Sacks, \emph{Reconstruction of a radially symmetric potential from two spectral sequences}, J. Math. Anal. Appl. \textbf{264} (2001), 354--381.

\bibitem{Ser07} F.~Serier, \emph{The inverse spectral problem for radial Schr\"odinger operators on $[0,1]$}, J.~Differential Equations \textbf{235} (2007), 101--126. 

\bibitem{Sh} C. Shubin Christ,  \emph{An inverse problem for the Schr\"odinger equation with a radial potential}, J. Differential Equations, 103 (1993), 247--259.



\bibitem{Sommerfeld12} A.~Sommerfeld, \emph{Die Greensche Funktion der Schwingungsgleichung}, Jahresber. Dtsch. Math.-Ver. \textbf{21} (1912), 309--353.

\bibitem{Szasz16} O.~Sz\'asz, \emph{\"Uber die Approximation stetiger Funktionen durch lineare Aggregate von Potenzen}, Math. Ann. \textbf{77} (1916), 482--496.


\bibitem{Walter98} W.~Walter, \textit{Ordinary Differential Equations}, Graduate Texts in Mathematics, vol.~182, Springer-Verlag, New York, 1998.

\bibitem{Wasow65}
W.\ Wasow, \emph{Asymptotic Expansions for Ordinary Differential Equations}, Pure and Applied Mathematics, Vol. XIV, Interscience Publishers (John Wiley \& Sons, Inc.), New York, London, Sydney, 1965 (reprinted by Dover, 1987).


\bibitem{Wa44} G.~N.~Watson, \emph{A Treatise on the Theory of Bessel Functions}, 2nd Edition, Cambridge University Press, 1944.

	
	
	
\end{thebibliography}
\end{document}